\documentclass[letterpaper,12pt]{article}
\usepackage[margin=1.2in]{geometry}
\usepackage{enumerate}

\usepackage{amssymb}
\usepackage{amsmath}
\usepackage{amsthm}
\usepackage{eufrak}
\usepackage{mathrsfs}

\usepackage{tikz}
\usetikzlibrary{calc}
\usepackage{subcaption}
\definecolor{red1}{RGB}{230,25,75}

\usepackage{cleveref}  
\crefname{graph}{Graph}{Graphs}

\newtheorem{theorem}{Theorem}[section]
\newtheorem{proposition}[theorem]{Proposition}
\newtheorem{corollary}[theorem]{Corollary}
\newtheorem{lemma}[theorem]{Lemma}

\newtheorem{open}[theorem]{Open Problem}

\usepackage[isbn=false,doi=false]{biblatex}

\bibliography{MetrizableStructure} 
\begin{document}
\title{The Structure of Metrizable Graphs}
\author{Maria Chudnovsky\thanks{Princeton University, Princeton, NJ 08544, USA. e-mail: mchudnov@math.princeton.edu.   {Supported by NSF-EPSRC Grant DMS-2120644 and by AFOSR grant FA9550-22-1-0083}} \and{Daniel Cizma\thanks{Einstein Institute of Mathematics, Hebrew University, Jerusalem 91904, Israel. e-mail: daniel.cizma@mail.huji.ac.il.}} \and  {Nati Linial\thanks{School of Computer Science and Engineering, Hebrew University, Jerusalem 91904, Israel. e-mail: nati@cs.huji.ac.il~. {Supported by NSF-BSF US-Israel Grant 2021690 "The Global Geometry of Graphs"}}}}

\maketitle

\begin{abstract}
    A {\em consistent path system} in a graph $G$ is an intersection-closed collection of paths, with exactly one path between any two vertices in $G$. We call $G$ {\em metrizable} if every consistent path system in it is the system of geodesic paths defined by assigning some positive lengths to its edges. We show that metrizable graphs are, in essence, subdivisions of a small family of basic graphs with additional compliant edges. In particular, we show that every metrizable graph with $11$ vertices or more is outerplanar plus one vertex.
\end{abstract}
\section{Introduction}
Let $G=(V,E)$ be
a connected graph, and let $w: E\to \mathbb{R}_{>0}$ be a positive weight function on its edges. This induces a metric on $V$, where the distance between any two vertices is the
least $w$-length of a path between them. What can be said about such a system of geodesics? E.g., what does the collection of $w$-geodesics tell us about $w$? Is it possibly true that {\em every} collection of paths in a graph constitute the system of geodesics corresponding to some graph metric? To simplify matters, suppose that $w$ is such that the shortest path between any two vertices is unique. Clearly, any subpath of a geodesic in $G$ is itself a geodesic. With this in mind, we define the notion of a {\em consistent path system} $\mathcal{P}$ in $G$. This is a collection of paths that is closed under taking subpaths, with a unique $uv$ path in $\mathcal{P}$ for each pair $u,v\in V$. So, we refine the question and ask if every consistent path system coincides with the set of geodesics that corresponds to some positive weight function on the edges. We already know \cite{CL} that 
this is not necessarily so, and here we seek to understand {\em metrizable} graphs. Namely, graphs in which {\em every} consistent path system is induced by some graph metric. For  questions of a similar flavor that come from physics see, e.g., \cite{M}.

The study of metrizable graphs was initiated in \cite{CL} where it was shown that metrizable graphs are in fact quite rare. For example, all large metrizable graphs are planar and not $3$-connected. On the other hand, that paper exhibits an infinite family of metrizable graphs, viz., all outerplanar graphs. In this paper we hone in on the structure of metrizable graphs. 
In particular, we show that every large $2$-connected metrizable graph can be obtained by taking some basic graph, subdividing its edges and iteratively adding edges between vertices connected by flat paths. (Recall a path in $G$ is said to be {\em flat} if every internal vertex in it has degree $2$ in $G$.) For instance, every $2$-connected outerplanar graph can be constructed using this procedure starting with a cycle. Explicitly, if $G$ is outerplanar,
then there exist graphs $G_0,G_1,\dots, G_t=G$ where $G_0$ is an cycle and $G_{i+1} = G_i\cup e_i$ is obtained from $G_i$ by adding an edge $e_i$ between two vertices which are connected by a flat path in $G_i$. We show here that every metrizable graph on $11$ vertices or more can be constructed by this procedure starting with $G_0=K_{2,n}$ or is else a subdivision of $K_{2,3},C_3, K_4, W_4, W_4'$, see \cref{fig:mainTheorem}. This, in particular, implies that every large metrizable graph can be made outerplanar by removing at most one vertex. 

\section{Preliminaries and Overview}
All graphs in this paper are finite and simple.  We recall some definitions and concepts from \cite{CL}. A {\em consistent path system} in a graph $G=(V,E)$ is a collection of paths $\mathcal{P}$ in $G$ satisfying two properties: 
 \begin{enumerate}[1)]
     \item For every $u,v\in V$ there is exactly one $uv$-path in $\mathcal{P}$
     \item Every two paths in $\mathcal{P}$ are either: (i) vertex disjoint, or (ii) have exactly one vertex in common, or (iii) their intersection is a path in $\mathcal{P}$.
 \end{enumerate}
A path system $\mathcal{P}$ of $G=(V,E)$ is {\em metric} if there is a positive weight function \mbox{$w:E\to \mathbb{R}_{>0}$} such that each path in $\mathcal{P}$ is a $w$-shortest path. We call $G$ {\em metrizable} if every consistent path system in $G$ is metric. It is known
\begin{proposition}[\cite{CL}]
The family of metrizable graphs is closed under topological minors.   
\end{proposition}
\noindent
(Recall that $H$ is a topological minor of $G$ if $G$ contains a subdivision of $H$ as a subgraph.) Therefore, every graph that contains a subdivision of a non-metrizable graph is itself non-metrizable. \\
The length of a path $P$, denoted $|P|$, is the number of its edges. A vertex is called a {\em branch vertex} if it has degree at least $3$. A {\em flat path}, aka a suspended path in $G$ is a path of length at least $2$ all of whose internal vertices are of degree $2$ in $G$. We call an edge $xy$ in $G$ {\em compliant} if $x$ and $y$ are also connected by a flat path. The role
of compliant edges in the present context is captured by the following 
result: 
\begin{proposition}[\cite{CL}]\label{compliant_can_be_deleted}
If $e$ is a compliant edge in $G$, then $G$ is metrizable if and only if $G-e$ is metrizable.
\end{proposition}
Our main result is that every metrizable graph can be constructed by subdividing some basic graph, 
\cref{fig:mainTheorem}, and adding compliant edges along its subdivided edges.
As usual, $W_4$ is the $4$-wheel. We denote the $4$-wheel plus one edge by $W_4'$. 
\begin{theorem}\label{thm:mainThm}
If a $2$-connected metrizable graph $G$ with at least $11$ vertices has
no compliant edges, then it is either $K_{2,n}$ for some $n\ge 4$ or a subdivision of one of the following: $K_{2,3}$, $K_4$, $W_4$ or $W_4'$.
\end{theorem}
\begin{figure}
    \centering
    	\begin{minipage}[c]{0.12\textwidth}
		\centering
	\begin{tikzpicture}[scale=0.4, every node/.style={scale=0.4}]
		\def\vtxSize{.4cm}
		\def\edgeWidth{1pt}
		\def\radius{2.5cm}
		
		\node[draw,circle,minimum size=\vtxSize,inner sep=1pt,fill, color=black] (1) at (-2,0) [scale=1] {};
		\node[draw,circle,minimum size=\vtxSize,inner sep=1pt,fill, color=black] (2) at (0,0) [scale=1]{};
		\node[draw,circle,minimum size=\vtxSize,inner sep=1pt,fill, color=black] (3) at (2,0)[scale=1] {};
		\node[draw,circle,minimum size=\vtxSize,inner sep=1pt,fill, color=black] (4) at (0,-2)[scale=1] {};
		\node[draw,circle,minimum size=\vtxSize,inner sep=1pt,fill, color=black] (5) at (0,2)[scale=1] {};

		\draw [line width=\edgeWidth,-] (1) -- (4);
		\draw [line width=\edgeWidth,-] (2) -- (4);
		\draw [line width=\edgeWidth,-] (3) -- (4);
		\draw [line width=\edgeWidth,-] (1) -- (5);
		\draw [line width=\edgeWidth,-] (2) -- (5);
		\draw [line width=\edgeWidth,-] (3) -- (5);

	\end{tikzpicture}
	\end{minipage}
        \hspace{5mm}
		\begin{minipage}[c]{0.12\textwidth}
			\centering
		\begin{tikzpicture}[scale=0.4, every node/.style={scale=0.4}]
			\def\vtxSize{.4cm}
			\def\edgeWidth{1pt}
			\def\radius{2.5cm}
			
			\node[draw,circle,minimum size=\vtxSize,inner sep=1pt,fill, color=black ](1) at (-0*360/3 +90: \radius) [scale=1] {};
			\node[draw,circle,minimum size=\vtxSize,inner sep=1pt,fill, color=black] (2) at (-1*360/3 +90: \radius) [scale=1]{};
			\node[draw,circle,minimum size=\vtxSize,inner sep=1pt,fill, color=black] (3) at (-2*360/3 +90: \radius)[scale=1] {};

			\draw [line width=\edgeWidth,-] (1) -- (2);
			\draw [line width=\edgeWidth,-] (2) -- (3);
			\draw [line width=\edgeWidth,-] (1) -- (3);

		\end{tikzpicture}
	\end{minipage}
        \hspace{5mm}
	\begin{minipage}[c]{0.12\textwidth}
		\centering
	\begin{tikzpicture}[scale=0.4, every node/.style={scale=0.4}]
		\def\vtxSize{.4cm}
		\def\edgeWidth{1pt}
		\def\radius{2.5cm}
		
		\node[draw,circle,minimum size=\vtxSize,inner sep=1pt,fill, color=black ](1) at (-0*360/3 +90: \radius) [scale=1] {};
		\node[draw,circle,minimum size=\vtxSize,inner sep=1pt,fill, color=black] (2) at (-1*360/3 +90: \radius) [scale=1]{};
		\node[draw,circle,minimum size=\vtxSize,inner sep=1pt,fill, color=black] (3) at (-2*360/3 +90: \radius)[scale=1] {};
		\node[draw,circle,minimum size=\vtxSize,inner sep=1pt,fill, color=black] (4) at (0,0)[scale=1] {};

		\draw [line width=\edgeWidth,-] (1) -- (2);
		\draw [line width=\edgeWidth,-] (1) -- (3);
		\draw [line width=\edgeWidth,-] (1) -- (4);
		\draw [line width=\edgeWidth,-] (2) -- (3);
		\draw [line width=\edgeWidth,-] (2) -- (4);
		\draw [line width=\edgeWidth,-] (3) -- (4);
		
	\end{tikzpicture}
\end{minipage}
\hspace{5mm}
\begin{minipage}[c]{0.12\textwidth}
	\centering
\begin{tikzpicture}[scale=0.4, every node/.style={scale=0.4}]
	\def\vtxSize{.4cm}
	\def\edgeWidth{1pt}
	\def\radius{2.5cm}
	
	\node[draw,circle,minimum size=\vtxSize,inner sep=1pt,fill, color=black ](1) at (-0*360/4 +45 : \radius) [scale=1] {};
	\node[draw,circle,minimum size=\vtxSize,inner sep=1pt,fill, color=black] (2) at (-1*360/4 +45: \radius) [scale=1]{};
	\node[draw,circle,minimum size=\vtxSize,inner sep=1pt,fill, color=black] (3) at (-2*360/4 +45: \radius)[scale=1] {};
	\node[draw,circle,minimum size=\vtxSize,inner sep=1pt,fill, color=black] (4) at (-3*360/4 +45: \radius)[scale=1] {};
	\node[draw,circle,minimum size=\vtxSize,inner sep=1pt,fill, color=black] (5) at (0,0)[scale=1] {};

	\draw [line width=\edgeWidth,-] (1) -- (2);
	\draw [line width=\edgeWidth,-] (1) -- (4);
	\draw [line width=\edgeWidth,-] (1) -- (5);
	\draw [line width=\edgeWidth,-] (2) -- (3);
	\draw [line width=\edgeWidth,-] (2) -- (5);
	\draw [line width=\edgeWidth,-] (3) -- (4);
	\draw [line width=\edgeWidth,-] (3) -- (4);
	\draw [line width=\edgeWidth,-] (3) -- (5);
	\draw [line width=\edgeWidth,-] (4) -- (5);
	
\end{tikzpicture}
\end{minipage}
\hspace{5mm}
\begin{minipage}[c]{0.12\textwidth}
	\centering
\begin{tikzpicture}[scale=0.4, every node/.style={scale=0.4}]
	\def\vtxSize{.4cm}
	\def\edgeWidth{1pt}
	\def\radius{2.1cm}
	
	\node[draw,circle,minimum size=\vtxSize,inner sep=1pt,fill, color=black](1) at (0*360/5 +162 : \radius) [scale=1] {};
	\node[draw,circle,minimum size=\vtxSize,inner sep=1pt,fill, color=black] (2) at (1*360/5 +162: \radius) [scale=1]{};
	\node[draw,circle,minimum size=\vtxSize,inner sep=1pt,fill, color=black] (3) at (2*360/5 +162: \radius)[scale=1] {};
	\node[draw,circle,minimum size=\vtxSize,inner sep=1pt,fill, color=black] (4) at (3*360/5 +162: \radius)[scale=1] {};
	\node[draw,circle,minimum size=\vtxSize,inner sep=1pt,fill, color=black] (5) at (4*360/5 +162: \radius)[scale=1] {};

	\draw [line width=\edgeWidth,-] (1) -- (2);
	\draw [line width=\edgeWidth,-] (1) -- (3);
	\draw [line width=\edgeWidth,-] (1) -- (4);
	\draw [line width=\edgeWidth,-] (1) -- (5);
	\draw [line width=\edgeWidth,-] (2) -- (4);
	\draw [line width=\edgeWidth,-] (2) -- (5);
	\draw [line width=\edgeWidth,-] (3) -- (4);
	\draw [line width=\edgeWidth,-] (3) -- (5);
	\draw [line width=\edgeWidth,-] (4) -- (5);
	
	\end{tikzpicture}
\end{minipage}
    \caption{Up to adding compliant edges, metrizable graphs are either $K_{2,n}$ or subdivisions of one of the above graphs.}
    \label{fig:mainTheorem}
\end{figure}
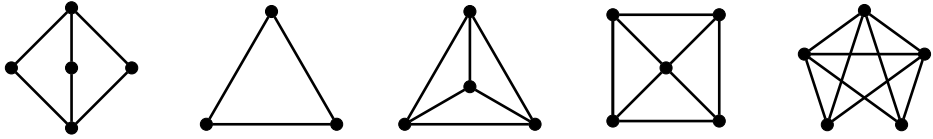
\noindent
The proof of this theorem builds on the very restricted structure of metrizable graphs.
\begin{theorem}\label{lem:disjointCycles}
If a graph $G$ with at least $11$ vertices is
(i) $2$-connected, (ii) has no compliant edges, (iii) has at least two disjoint cycles,
then $G$ is non-metrizable.
\end{theorem}
The basic method for showing that a graph $G=(V,E)$ is non-metrizable,
was developed in \cite{CL}: Namely, we
consider a consistent path system $\mathcal{P}$ in $G$. 
Associated with $\mathcal{P}$ is a system
of linear inequalities, and $\mathcal{P}$ is
metric iff this system is feasible. So if the chosen $\mathcal{P}$
is non-metric, we can use
LP-duality to create a {\em hand-checkable certificates} of this.
Thus, using a computer, we created
a ``zoo'' of 16 non-metrizable graphs along with such path systems and the
corresponding certificates, \cref{fig:zoo} and 
\cref{append:certificates}. We will refer throughout to graphs from the zoo as
7a-7p, as they appear in \cref{fig:zoo}. Our proof of \cref{lem:disjointCycles} 
shows that any graph satisfying these conditions
contains a subdivision of some graph from the zoo.
Our argument splits according to whether $G$ contains two disjoint cycles or not.
If it does, then we use these cycles and paths between them
to find subdivisions of zoo graphs. On the other hand,
as shown by Lov\'asz (\cref{lem:Lovasz}), graphs with no two disjoint cycles
have a very restricted structure. This allows us to derive
a proof of \cref{thm:mainThm} from \cref{lem:disjointCycles}.
Finally, we show that every metrizable graph is nearly outerplanar.
\begin{corollary}\label{cor:outerplanar}
Every $2$-connected metrizable graph with at least $11$ vertices can be made outerplanar by removing at most one vertex.
\end{corollary}
This corollary follows from \cref{thm:mainThm} and the fact that outerplanarity is 
preserved under the addition of compliant edges (\cref{compliant_can_be_deleted}). Note that the statement of \cref{cor:outerplanar}
need not hold for smaller graphs. For example, $K_6$ is metrizable but can be made 
outerplanar only by removing $3$ vertices. 

This paper is organized as follows. In section 3 we prove several lemmas
needed for our main proofs. In section 4 we prove our main results: \cref{lem:disjointCycles}, \cref{thm:mainThm} and \cref{cor:outerplanar}. 
Section 5 is devoted to open questions.
Following section 5 is a figure with our zoo of non-metrizable graphs, and
appendix A contains certificates of their non-metrizability.
\section{Lemmas}

We begin with some relevant preliminary facts:
\begin{proposition}
\label{prop:compliantEdges}
Let $G=(V,E)$ be a graph which is not a cycle.
\begin{enumerate}
    \item If $G$ is $2$-connected and outerplanar then it has a compliant edge. In particular, $G$ has at least two edge disjoint flat paths whose endpoints span a compliant edge.\label{part:compliant}
    \item If $G-e$ is outerplanar for some compliant edge $e \in E(G)$, then $G$ is outerplanar.\label{part:outerplanar}
    \item If $G$ is $2$-connected and $e \in E(G)$ is a compliant edge, then $G-e$ is $2$-connected.\label{part:2conn}

\end{enumerate}
\end{proposition}

\begin{proof}
\begin{enumerate}
\item
The edge set of
every $2$-connected outerplanar graph has the form $C\sqcup M$, with $C$ 
a Hamiltonian cycle and $M$ cyclically non-crossing edges. 
We argue by induction on $|M|$. 
If $|M|=1$ then $G$ is a cycle with one additional edge $xy$ which splits $C$ into two edge 
disjoint flat $xy$ paths. For the induction step, pick an edge $xy \in M$. 
By the induction hypothesis $G-xy$ has two edge disjoint flat paths $P_1$ and and $P_2$ whose endpoints span a compliant edge. If neither $x$ nor $y$ is an internal vertex of $P_1$ or $P_2$ then the same two paths work for $G$ as well.

Otherwise, say $x$ is an internal vertex of $P_1$, a flat path in $G-e$ between $u$ and $v$,
where, by assumption $uv \in M$.
We conclude that also $y$ belongs to $P_1$, for otherwise
$C \cup uv \cup xy$ forms a subdivided $K_4$, contrary to $G$ being outerplanar. 
Now consider the $xy$ subpath $P_1'$ of $P_1$.
Clearly, $P_1'$ is a flat path in $G$ whose ends span the compliant edge $xy$.
The paths $P_1'$ and $P_2$ satisfy the claim.
\item  
Suppose toward a contradiction that $G$ is not outerplanar, while $G-e$ is, where $e=xy$
is a compliant edge
that connects between the two ends of the flat path $P$. Since $G$ is not outerplanar,
it must have a subgraph $H$ that is either a subdivided $K_4$
or a subdivided $K_{2,3}$, and since $G-e$ is outerplanar, $e$ must belong to $E(H)$.
If $x,y$ are the only common vertices of $P$ and $H$, then by removing the edge $e$
from $H$ and adding instead the path $P$, we obtain a subdivision of $H$,
which is impossible, since $G-e$ is outerplanar.

But $P$ is flat and has no branch vertices other that possibly $x,y$, so that if $P$ and $H$
have a non-trivial intersection, i.e. $H$ contains an internal vertex of $P$, then necessarily $P\subseteq H$.
The cycle $C = P\cup xy$ is contained in $H$. If $H$ is a subdivided $K_4$, then each of its
cycles has $3$ branch vertices, while $C$ has no branch vertices other than
(possibly) $x$ and $y$. Similarly, if $H$ is a subdivision of $K_{2,3}$ then 
$C$ must have two non-adjacent branch vertices, again a contradiction.
\item
Again let $P$ be a flat $xy$ path, and $e=xy$ the corresponding compliant edge.
Since $G$ is not a cycle there must be another vertex $z\in G - C$,
where $C = P\cup e$.
Since $G$ is $2$-connected, there exist two paths $Q_1,Q_2$ from 
$z$ to $C$ which are disjoint except at $z$. Since $P$ is flat, $Q_1$ and $Q_2$ 
must go from $z$ to $\{x,y\}$, respectively.
To see that $G-e$ is $2$-connected we exhibit two internally disjoint $xy$-paths in $G-e$,
namely $P$ and the concatenation of $Q_1$ and $Q_2$.
\end{enumerate}
\end{proof}
Thus \cref{part:2conn} of \cref{prop:compliantEdges} implies that, 
given any $2$-connected $G$, we can iteratively remove compliant edges, which process
ends with a cycle or a $2$-connected subgraph with no compliant edges.



To prove our main results we need some groundwork, starting with the following technical lemma regarding compliant edges in $2$-connected graphs:
\begin{lemma} \label{lem:compliantOuterplanar}
Let $G=(V,E)$ be a $2$-connected graph which is not a cycle, and which has no compliant edges. 
Let $e=xy\in E$, let $S$ be the vertex set of a connected component of $G-\{x,y\}$, and let 
$H$ be the subgraph spanned by $S\cup \{x,y\}$. Then $H$ is $2$-connected, and not outerplanar.
\end{lemma}
\begin{proof}
Since $G$ is $2$-connected and has no compliant edges, it is non-outerplanar
by \cref{prop:compliantEdges}. This yields our claim when $H=G$,
so we can and will assume henceforth that $H\neq G$.
That $H$ is $2$-connected follows from a standard argument. For a vertex $z\in H$ we need to 
show that $H-z$ is connected. For any two vertices $u,v\in H-z$ we find a $(uv)$-path in $H-z$. Since $G$ is $2$-connected, $G-z$ is connected 
and there exists a path $P\subseteq G-z$ between $u$ and $v$. Since $x$ and $y$ separates $H$ from the rest of $G$, either $P\subseteq H-z$ or $x,y\in P$ with, say,
the $(ux)$-subpath and the $(vy)$-subpath of $P$ both contained in $H$. In the latter case, we can concatenate these two subpaths with edge $xy$ to obtain a $(uv)$-path in $H-z$. \\
We now show $H$ is not outerplanar. If $H$ is a cycle, then
$P = H -xy$ is a flat path in $H$ with $xy$ the corresponding compliant edge.
But $P$ is flat in $G$ as well, because other than $x$ and $y$ no vertex
in $P$ has neighbors outside of $H$. However, by
assumption $G$ has no compliant edges, 
so we can and will assume that $H$ is not a cycle. If $H$ is outerplanar, then by
\cref{prop:compliantEdges}, $H$ has two edge disjoint flat paths $P_1, P_2$ whose endpoints 
form compliant edges. By our assumption these paths are not flat in $G$.
Since $x$ and $y$ are the only vertices in $H$ that have neighbors outside of $H$,
necessarily $x$ and $y$ are internal vertices of $P_1$ and $P_2$. But then $P_1$ cannot be flat in $H$, because $x$
has a neighbor outside of $P_1$, namely $y$.
\end{proof}

\noindent
Next, we make a simple observation on the structure of non-outerplanar graphs.

\begin{lemma}\label{lem:nonOuterStucture}
If $G$ is a $2$-connected graph which is not outerplanar, then
every edge of $G$ is in some
subgraph which is a subdivision of $K_4$ or of $K_{2,3}$.
\end{lemma}
\begin{proof}
Let $H$ be a subgraph of $G$ which is a subdivision of 
$K_4$ or $K_{2,3}$. Such $H$ exists, since $G$ is not outerplanar.
Let $xy\in E(G)$, and wlog $x\notin V(H)$. We
proceed according to whether $y\in V(H)$ or not.

Consider first the case $y\in H$. Since $G$ is $2$-connected, there 
is a path $P$ from $x$ to a vertex $u\in H$, $u\neq y$. 
But then $yxP$ is a $yu$ path whose internal vertices are disjoint from $H$. 
If $u$ and $y$ are neighbors, we may replace the edge $yu\in E(H)$ with the $yu$-path $yxP$ 
to obtain a graph isomorphic to a subdivision of $H$ containing both vertices $x$ and $y$. 
If $uy\not\in E(H)$ then there exists two internally disjoint $yu$ paths in $H$, 
of length at least $2$. These two $yu$ paths along with the path $yxP$ constitute a subdivided
$K_{2,3}$ containing $x$ and $y$.

Next suppose that neither $x$ nor $y$ is in $H$. 
Since $G$ is $2$-connected, there exist two disjoint paths between $\{x,y\}$ and $H$. 
Namely, $P_1$ is a $xu$ path, $P_2$ is a $yv$ path. Two
paths that are disjoint from $H$ except at their endpoints $u$ and $v$. The
concatenated path $Q=P_1^{-1}xyP_2$, is a $uv$ path whose 
internal vertices are disjoint from $H$.
As before, if $uv \in E(H)$ then we may replace the edge $uv$ by the path $Q$ to obtain a
subgraph of $G$ that is isomorphic to a
subdivision of $H$. Similarly, if $uv \not\in E(H)$ then there are two internally disjoint $uv$ 
paths in $H$, with length at least $2$. These two paths together with $Q$ form
a subdivided $K_{2,3}$ that contains both $x$ and $y$.
\end{proof} 
Let $K_{2,4}'$  denote the graph obtained by adding to $K_{2,4}$ an edge between two vertices on the side of $4$. We prove:

\begin{lemma}\label{lem:K24}
    If a $2$-connected graph $G$ with at least $9$ vertices has a subgraph $H$ that is a subdivision of $K_{2,4}'$, then $G$ is not metrizable.
\end{lemma}
\begin{proof}
Since $H$ contains a subdivision of $K_{2,4}$, there exist
$x,y\in H$ with $4$ internally disjoint $xy$ paths $P_1,P_2,P_3,P_4$ between them of
length at least $2$. Observe that if the length of any $P_i$ strictly exceeds $2$,
then we get a subdivision of \cref{fig:graph1}. 
We can therefore assume that $H$ looks as follows: 
The above paths are $P_i = xu_iy$, $i=1,2,3,4$ 
and in addition there is a $u_1u_2$ path $Q$, disjoint from $\{x,y,u_3,u_4\}$. 
Note that if $|Q|>2$ then we obtain a subdivision of \cref{fig:graph11}. 
Therefore, we may assume that $Q$ is either an edge or has length $2$. 
We argue the case where $Q$ is an edge, \cref{fig:K24'notMet1}, and note the same argument
applies as well when $Q$ has length $2$. 
Since $H$ has $6$ vertices there is a vertex $z\in G - H$, and since
$G$ is $2$-connected, there are two paths from $z$ to $H$ which are disjoint except at $z$.
Concatenating these paths at $z$, we obtain a path of length at least $2$ which intersects 
$H$ only at its endpoints. If these endpoints are $x$ and $y$,
then we obtain a subdivision of \cref{fig:graph1}, see \cref{fig:K24'notMet2}. Similarly, if the endpoints of this path is $x$ and $u_i$ (or $y$ and $u_i$) for some $i=1,2,3,4$ we again obtain a subdivision of \cref{fig:graph1}. Therefore, the only possibility left is that the path connects two vertices $u_i, u_j$ for $i\neq j$. (For the following arguments we may ignore the edge $u_1u_2$ so that the argument is symmetric w.r.t.\
vertices $u_1,u_2,u_3,u_4$.)
\begin{figure}
		\centering
		\begin{subfigure}{0.28\textwidth}
			\centering
			\begin{tikzpicture}
                \def\width{1.75cm}
			\def\vtxSize{0.5cm}
			\def\height{1.5cm}
                \def\edgewidth{1.25pt}
			\node[draw,circle,minimum size=\vtxSize,inner sep=0pt] (u1) at (-\width,0) {$u_1$};
                \node[draw,circle,minimum size=\vtxSize,inner sep=0pt] (u4) at (\width,0) {$u_4$};
                \node[draw,circle,minimum size=\vtxSize,inner sep=0pt] (u2) at  ($(u1)!0.33!(u4)$){$u_2$};
                \node[draw,circle,minimum size=\vtxSize,inner sep=0pt] (u3) at  ($(u1)!0.66!(u4)$){$u_3$};

			\node[draw,circle,minimum size=\vtxSize,inner sep=0pt] (x) at (0,\height) {$x$};
			\node[draw,circle,minimum size=\vtxSize,inner sep=0pt] (y) at (0,-\height) {$y$};

			\draw [line width=\edgewidth,-] (u1) -- (u2);
                \draw [line width=\edgewidth,-] (u1) -- (x);
                \draw [line width=\edgewidth,-] (u2) -- (x);
                \draw [line width=\edgewidth,-] (u3) -- (x);
                \draw [line width=\edgewidth,-] (u4) -- (x);
                \draw [line width=\edgewidth,-] (u1) -- (y);
                \draw [line width=\edgewidth,-] (u2) -- (y);
                \draw [line width=\edgewidth,-] (u3) -- (y);
                \draw [line width=\edgewidth,-] (u4) -- (y);

			\end{tikzpicture}
			\caption{A $K_{2,4}'$ subgraph}
			\label{fig:K24'notMet1}
			
		\end{subfigure}
		\hfill
		\begin{subfigure}{0.28\textwidth}
			\centering
			\begin{tikzpicture}
                \def\width{1.75cm}
			\def\vtxSize{0.5cm}
			\def\height{1.5cm}
                \def\edgewidth{1.25pt}
			\node[draw,circle,minimum size=\vtxSize,inner sep=0pt] (u1) at (-\width,0) {$u_1$};
                \node[draw,circle,minimum size=\vtxSize,inner sep=0pt] (z) at (\width,0) {$z$};
                \node[draw,circle,minimum size=\vtxSize,inner sep=0pt] (u2) at  ($(u1)!0.25!(z)$){$u_2$};
                \node[draw,circle,minimum size=\vtxSize,inner sep=0pt] (u3) at  ($(u1)!0.5!(z)$){$u_3$};
                \node[draw,circle,minimum size=\vtxSize,inner sep=0pt] (u4) at  ($(u1)!0.75!(z)$){$u_4$};

			\node[draw,circle,minimum size=\vtxSize,inner sep=0pt] (x) at (0,\height) {$x$};
			\node[draw,circle,minimum size=\vtxSize,inner sep=0pt] (y) at (0,-\height) {$y$};

			\draw [line width=\edgewidth,-,red1] (u1) -- (u2);
                \draw [line width=\edgewidth,-] (u1) -- (x);
                \draw [line width=\edgewidth,-,red1] (u2) -- (x);
                \draw [line width=\edgewidth,-,red1] (u3) -- (x);
                \draw [line width=\edgewidth,-,red1] (u4) -- (x);
                \draw [line width=\edgewidth,-,red1] (u1) -- (y);
                \draw [line width=\edgewidth,-] (u2) -- (y);
                \draw [line width=\edgewidth,-,red1] (u3) -- (y);
                \draw [line width=\edgewidth,-,red1] (u4) -- (y);
                \draw [line width=\edgewidth,dashed,red1] (z) -- (y);
                \draw [line width=\edgewidth,dashed,red1] (z) -- (x);

			\end{tikzpicture}
			\caption{A subdivision of \cref{fig:graph1}.}
			\label{fig:K24'notMet2}
			
		\end{subfigure}
		\hfill
		\begin{subfigure}{0.28\textwidth}
			\centering
			\begin{tikzpicture}
                \def\width{1.75cm}
			\def\vtxSize{0.5cm}
			\def\height{1.5cm}
                \def\edgewidth{1.25pt}
                \def\shift{.5cm}
			\node[draw,circle,minimum size=\vtxSize,inner sep=0pt] (u1) at (-\width,0) {$u_1$};
                \node[draw,circle,minimum size=\vtxSize,inner sep=0pt] (u4) at (\width,0) {$u_4$};
                \node[draw,circle,minimum size=\vtxSize,inner sep=0pt] (u2) at  ($(u1)!0.6!(u4)$){$u_2$};
                \node[draw,circle,minimum size=\vtxSize,inner sep=0pt] (u3) at  ($(u1)!0.8!(u4)$){$u_3$};
                 \node[draw,circle,minimum size=.5*\vtxSize,inner sep=1pt,fill, color=black] (z) at  ($(u1)!0.5!(u2)$){};
                \node[draw,circle,minimum size=.5*\vtxSize,inner sep=1pt,fill, color=black] (z') at  ($(u1)!0.5!(u2) + (0,\shift)$){};
                \node[draw,circle,minimum size=.5*\vtxSize,inner sep=1pt,fill, color=black] (z'') at  ($(u1)!0.5!(u2) + (0,-\shift)$){};

			\node[draw,circle,minimum size=\vtxSize,inner sep=0pt] (x) at (0,\height) {$x$};
			\node[draw,circle,minimum size=\vtxSize,inner sep=0pt] (y) at (0,-\height) {$y$};

			\draw [line width=\edgewidth,-,red1] (u1) -- (z);
                \draw [line width=\edgewidth,-,red1] (u1) -- (z');
                \draw [line width=\edgewidth,-,red1] (u1) -- (z'');
                \draw [line width=\edgewidth,-] (u1) -- (x);
                \draw [line width=\edgewidth,-,red1] (u2) -- (x);
                \draw [line width=\edgewidth,-,red1] (u2) -- (z);
                \draw [line width=\edgewidth,-,red1] (u2) -- (z');
                \draw [line width=\edgewidth,-,red1] (u2) -- (z'');
                \draw [line width=\edgewidth,-] (u3) -- (x);
                \draw [line width=\edgewidth,-,red1] (u4) -- (x);
                \draw [line width=\edgewidth,-,red1] (u1) -- (y);
                \draw [line width=\edgewidth,-] (u2) -- (y);
                \draw [line width=\edgewidth,-] (u3) -- (y);
                \draw [line width=\edgewidth,-,red1] (u4) -- (y);

			\end{tikzpicture}
			\caption{A subdivision of \cref{fig:graph1}.}
			\label{fig:K24'notMet3}
			
		\end{subfigure}
		\caption{}
	\end{figure}
If the length of this $u_iu_j$ path is strictly larger than $2$ we obtain a subdivision of \cref{fig:graph11}. Therefore, there is a path $u_izu_j$ in $G$. Since $|H|=6$ and $|G|\geq 9$ there must be at least two other vertices $z'$ and $z''$ not in $H$, and the same argument 
yields paths $u_{i'}z'u_{j'}$ and $u_{i''}z''u_{j''}$  in $G$. We claim that if, say, 
$|\{i,j\}\cap \{i',j'\}|=1$ then $G$ is not metrizable. 
Indeed, if say $i=1$, $j=j'=2$ and $i'=3$ then \cref{fig:graph12} is a subgraph
of $H\cup u_1zu_2\cup u_2z'u_3$. So we may assume that the sets $\{i,j\}$,$\{i',j'\}$, $\{i'',j''\} \subset \{1,2,3,4\}$ are pairwise equal or disjoint. There are two subcases to consider: i) all three sets coincide, ii) two sets are equal while the third is disjoint. For i) suppose the three sets are equal to $\{1,2\}$. In this case $G$ contains a subdivision of \cref{fig:graph1}, see \cref{fig:K24'notMet3}. For ii) suppose two of sets are equal to $\{1,2\}$ while the third is $\{3,4\}$. Then $G$ contains a subdivision of \cref{fig:graph10}, see \cref{fig:K24'notMet4}.
In either case, $G$ is not metrizable.
\end{proof}

\section{Proof of the Main Results}
We now prove the main results starting with \cref{lem:disjointCycles}.

\begin{proof}{ [\Cref{lem:disjointCycles}]}
Let $C_1$ and $C_2$ be two disjoint cycles in $G$. Since $G$ is $2$-connected, there exists two disjoint paths, $R_1$ and 
$R_2$, (possibly edges) between them. Say $R_1$ connects between $x_1\in C_1$ to $x_2\in C_2$, 
and $R_2$ connects $y_1\in C_1$ to $y_2\in C_2$.\\
Set $H = C_1\cup C_2 \cup R_1\cup R_2$, \cref{fig:disCycCase1_1}.
There are three cases to consider:
\begin{itemize}
    \item[]  Case 1. Neither $x_1y_1$ nor $x_2y_2$ are edges of $C_1$ and $C_2$, respectively.
    \item[] Case 2. $x_1,y_1$ are not adjacent in $C_1$ but $x_2y_2$ is an edge of $C_2$
    \item[]  Case 3. Both $x_1y_1$ and $x_2y_2$ are adjacent in $C_1$ and $C_2$, respectively
\end{itemize} 
 \begin{figure}[h]
		\centering
		\begin{subfigure}[t]{0.28\textwidth}
			\centering
			\begin{tikzpicture}
                \def\width{1.75cm}
			\def\vtxSize{0.5cm}
			\def\height{1.5cm}
                \def\edgewidth{1.25pt}
                \def\r{\width/3}
			\node[draw,circle,minimum size=\vtxSize,inner sep=0pt] (u1) at (-\width,0) {$u_1$};
                \node[draw,circle,minimum size=\vtxSize,inner sep=0pt] (u4) at (\width,0) {$u_4$};
                \node[draw,circle,minimum size=\vtxSize,inner sep=0pt] (u2) at  ($(u1)!0.45!(u4)$){$u_2$};
                \node[draw,circle,minimum size=\vtxSize,inner sep=0pt] (u3) at  ($(u1)!0.65!(u4)$){$u_3$};

			\node[draw,circle,minimum size=\vtxSize,inner sep=0pt] (x) at (0,\height) {$x$};
			\node[draw,circle,minimum size=\vtxSize,inner sep=0pt] (y) at (0,-\height) {$y$};

                \node[draw,circle,minimum size=.5*\vtxSize,inner sep=1pt,fill, color=black] (z) at  ($(u1)!.5!(u2)+ (0,\r/2)$){};
                \node[draw,circle,minimum size=.5*\vtxSize,inner sep=1pt,fill, color=black] (z') at  ($(u1)!.5!(u2)+ (0,-\r/2)$){};
			\node[draw,circle,minimum size=.5*\vtxSize,inner sep=1pt,fill, color=black] (z'') at  ($(u3)!.5!(u4)$){};

                \draw [line width=\edgewidth,-,red1] (u1) -- (x);
                \draw [line width=\edgewidth,-,red1] (u2) -- (x);
                \draw [line width=\edgewidth,-] (u3) -- (x);
                \draw [line width=\edgewidth,-,red1] (u4) -- (x);
                \draw [line width=\edgewidth,-,red1] (u1) -- (y);
                \draw [line width=\edgewidth,-,red1] (u2) -- (y);
                \draw [line width=\edgewidth,-,red1] (u3) -- (y);
                \draw [line width=\edgewidth,-] (u4) -- (y);
                \draw [line width=\edgewidth,-,red1] (u1) -- (z);
                \draw [line width=\edgewidth,-,red1] (u1) -- (z');
                \draw [line width=\edgewidth,-,red1] (u2) -- (z);
                \draw [line width=\edgewidth,-,red1] (u2) -- (z');
                \draw [line width=\edgewidth,-,red1] (u3) -- (z'');
                \draw [line width=\edgewidth,-,red1] (u4) -- (z'');

			\end{tikzpicture}
			\caption{A subdivision of \cref{fig:graph10}.}
			\label{fig:K24'notMet4}
			
		\end{subfigure}
		\hfill 
            \begin{subfigure}[t]{0.28\textwidth}
			\centering
			\begin{tikzpicture}
			\def\vtxSize{0.5cm}
                \def\width{1.25cm}
			\def\height{1.25cm}
                \def\edgewidth{1.5pt}
			\node[draw,circle,minimum size=\vtxSize,inner sep=0pt] (x1) at (-\width,\height) {$x_1$};
                \node[draw,circle,minimum size=\vtxSize,inner sep=0pt] (y1) at (-\width,-\height) {$y_1$};
                \node[draw,circle,minimum size=\vtxSize,inner sep=0pt] (x2) at  (\width, \height){$x_2$};
                \node[draw,circle,minimum size=\vtxSize,inner sep=0pt] (y2) at  (\width, -\height){$y_2$};
                \node at (-\width, 0) {$C_1$};
                \node at (\width, 0) {$C_2$};

			\draw [line width=\edgewidth,dotted] (x1) -- (x2) node[midway, above]{$R_1$};
                \draw [line width=\edgewidth,dotted] (y1) -- (y2) node[midway, above]{$R_2$};;
                \draw [line width=\edgewidth,dotted] (x1) to[out=-135, in=135] (y1);
                \draw [line width=\edgewidth,dotted] (x1) to[out=-45, in=45] (y1);
                \draw [line width=\edgewidth,dotted] (x2) to[out=-135, in=135] (y2);
                \draw [line width=\edgewidth,dotted] (x2) to[out=-45, in=45] (y2);

			\end{tikzpicture}
			\caption{The subgraph $H$ of disjoint cycles.}
			\label{fig:disCycCase1_1}
			
		\end{subfigure}
		\hfill
		\begin{subfigure}[t]{0.28\textwidth}
			\centering
			\begin{tikzpicture}
               \def\vtxSize{0.5cm}
                \def\width{1.25cm}
			\def\height{1.25cm}
                \def\edgewidth{1.25pt}
                \def\r{.5cm}
			\node[draw,circle,minimum size=\vtxSize,inner sep=0pt] (x1) at (-       \width,\height) {$x_1$};
                \node[draw,circle,minimum size=\vtxSize,inner sep=0pt] (y1) at (-\width,-\height) {$y_1$};
                \node[draw,circle,minimum size=\vtxSize,inner sep=0pt] (x2) at  (\width, \height){$x_2$};
                \node[draw,circle,minimum size=\vtxSize,inner sep=0pt] (y2) at  (\width, -\height){$y_2$};
                \node[draw,circle,minimum size=\vtxSize,inner sep=0pt] (u1) at  (-\width-\r, 0){$u_1$};
                \node[draw,circle,minimum size=\vtxSize,inner sep=0pt] (v1) at  (-\width+\r, 0){$v_1$};
                \node[draw,circle,minimum size=\vtxSize,inner sep=0pt] (u2) at  (\width-\r, 0){$u_2$};
                \node[draw,circle,minimum size=\vtxSize,inner sep=0pt] (v2) at  (\width+\r, 0){$v_2$};

			\draw [line width=\edgewidth,-,red1] (x1) -- (x2);
                \draw [line width=\edgewidth,-,red1] (y1) -- (y2);
                \draw [line width=\edgewidth,-,red1] (x1) -- (u1);
                \draw [line width=\edgewidth,-,red1] (x1) -- (v1);
                \draw [line width=\edgewidth,-,red1] (y1) -- (u1);
                \draw [line width=\edgewidth,-,red1] (y1) -- (v1);
                \draw [line width=\edgewidth,-] (x2) -- (u2);
                \draw [line width=\edgewidth,-,red1] (x2) -- (v2);
                \draw [line width=\edgewidth,-,red1] (y2) -- (u2);
                \draw [line width=\edgewidth,-,red1] (y2) -- (v2);
                \draw [line width=\edgewidth,dashed,red1] (x1) -- (u2);

			\end{tikzpicture}
			\caption{A subdivision of \cref{fig:graph2}.}
			\label{fig:disCycCase1_2}
			
		\end{subfigure}
		\caption{}
	\end{figure}
{\it Case 1:}
The vertices $x_i, y_i$ split $C_i$ into two arcs ($i=1,2$).
If any of these four arcs is strictly longer than $2$, then
we we obtain a subdivision of \cref{fig:graph14}. 
So we can assume that $C_1, C_2$ are $4$-cycles, with 
$x_i,y_i$ antipodal in $C_i$.  Also, if either $|R_1|>1$ or $|R_2|>1$ we get a subdivision of \cref{fig:graph13}. 
So we can assume that $R_1, R_2$ are in fact edges. 
Let $u_i,v_i$ be the two vertices in $C_i - \{x_i,y_i\}$. 
Since $G$ is $2$-connected and has at least $11$ vertices and $H$ contains $8$,
we can, as before, find an $H$-{\em ear} in $G$. Namely,
a path $Q$ of at least length $2$ 
between two vertices in $H$. We can assume that these two vertices are non-adjacent
since, as discussed above, subdividing any edge in $H$ yields a non-metrizable graph. 
If $Q$ runs between a vertex in $C_1$ and a non-neighbor of it in $C_2$, then
up to symmetry three subcases may occur:  $Q$ connects i) $x_1$ and $y_2$, ii) $x_1$ and $u_2$, 
and iii) $u_1$ and $u_2$. 
Subcases i) and ii) yield a subdivision of \cref{fig:graph2}, see
\cref{fig:disCycCase1_4}. Lastly, case iii) gives a subdivision of the prism,
\cref{fig:graph4}, which is not metrizable, see \cref{fig:disCycCase1_5}. So we may assume that $Q$ connects two non-adjacent vertices in the same $C_i$. 
Notice a path between $x_i$ and $y_i$ gives a subdivision of \cref{fig:graph1}. So the only option left is that $Q$ connects $u_i$ and $v_i$. In this case, if $|Q|>2$ we then obtain a subdivision of \cref{fig:graph16}, see \cref{fig:disCycCase1_7}, 
which leaves only the case $Q=u_i z v_i$. 
However, since $|H|=8$ and $|G|\geq 11$,
there exists at least three distinct vertices $z,z',z''\in G-H$ such that $u_{i}zv_{i}$, $u_{i'}z'v_{i'}$ and $u_{i''}z''v_{i''}$ are paths in $G$. Since $i,i',i'' \in \{1,2\}$ at least two of these indices must coincide, say $i=i'=1$, which yields
a subdivision of \cref{fig:graph10}, see \cref{fig:disCycCase1_3}, and $G$ is not metrizable.

  \begin{figure}[h]
		\centering
            \begin{subfigure}[t]{0.28\textwidth}
			\centering
			\begin{tikzpicture}
               \def\vtxSize{0.5cm}
                \def\width{1.25cm}
			\def\height{1.25cm}
                \def\edgewidth{1.25pt}
                \def\r{.5cm}
			\node[draw,circle,minimum size=\vtxSize,inner sep=0pt] (x1) at (-       \width,\height) {$x_1$};
                \node[draw,circle,minimum size=\vtxSize,inner sep=0pt] (y1) at (-\width,-\height) {$y_1$};
                \node[draw,circle,minimum size=\vtxSize,inner sep=0pt] (x2) at  (\width, \height){$x_2$};
                \node[draw,circle,minimum size=\vtxSize,inner sep=0pt] (y2) at  (\width, -\height){$y_2$};
                \node[draw,circle,minimum size=\vtxSize,inner sep=0pt] (u1) at  (-\width-\r, 0){$u_1$};
                \node[draw,circle,minimum size=\vtxSize,inner sep=0pt] (v1) at  (-\width+\r, 0){$v_1$};
                \node[draw,circle,minimum size=\vtxSize,inner sep=0pt] (u2) at  (\width-\r, 0){$u_2$};
                \node[draw,circle,minimum size=\vtxSize,inner sep=0pt] (v2) at  (\width+\r, 0){$v_2$};
			 \node[draw,circle,minimum size=.5*\vtxSize,inner sep=1pt,fill, color=black] (z) at  ($(x1)!.5!(u2)$){};

			\draw [line width=\edgewidth,-,red1] (x1) -- (x2);
                \draw [line width=\edgewidth,-,red1] (y1) -- (y2);
                \draw [line width=\edgewidth,-,red1] (x1) -- (u1);
                \draw [line width=\edgewidth,-,red1] (x1) -- (v1);
                \draw [line width=\edgewidth,-,red1] (y1) -- (u1);
                \draw [line width=\edgewidth,-,red1] (y1) -- (v1);
                \draw [line width=\edgewidth,-] (x2) -- (u2);
                \draw [line width=\edgewidth,-,red1] (x2) -- (v2);
                \draw [line width=\edgewidth,-,red1] (y2) -- (u2);
                \draw [line width=\edgewidth,-,red1] (y2) -- (v2);
                \draw [line width=\edgewidth,dashed,red1] (x1) -- (z);
                 \draw [line width=\edgewidth,dashed,red1] (u2) -- (z);

			\end{tikzpicture}
			\caption{A path between $x_1$ and $u_2$ yields a subdivision of \cref{fig:graph2}.}
			\label{fig:disCycCase1_4}
		\end{subfigure}
		\hfill
            \begin{subfigure}[t]{0.28\textwidth}
			\centering
			\begin{tikzpicture}
               \def\vtxSize{0.5cm}
                \def\width{1.25cm}
			\def\height{1.25cm}
                \def\edgewidth{1.25pt}
                \def\r{.5cm}
			\node[draw,circle,minimum size=\vtxSize,inner sep=0pt] (x1) at (-       \width,\height) {$x_1$};
                \node[draw,circle,minimum size=\vtxSize,inner sep=0pt] (y1) at (-\width,-\height) {$y_1$};
                \node[draw,circle,minimum size=\vtxSize,inner sep=0pt] (x2) at  (\width, \height){$x_2$};
                \node[draw,circle,minimum size=\vtxSize,inner sep=0pt] (y2) at  (\width, -\height){$y_2$};
                \node[draw,circle,minimum size=\vtxSize,inner sep=0pt] (u1) at  (-\width-\r, 0){$u_1$};
                \node[draw,circle,minimum size=\vtxSize,inner sep=0pt] (v1) at  (-\width+\r, 0){$v_1$};
                \node[draw,circle,minimum size=\vtxSize,inner sep=0pt] (u2) at  (\width-\r, 0){$u_2$};
                \node[draw,circle,minimum size=\vtxSize,inner sep=0pt] (v2) at  (\width+\r, 0){$v_2$};
			 \node[draw,circle,minimum size=.5*\vtxSize,inner sep=1pt,fill, color=black] (z) at  ($(v1)!.5!(u2)$){};

			\draw [line width=\edgewidth,-,red1] (x1) -- (x2);
                \draw [line width=\edgewidth,-,red1] (y1) -- (y2);
                \draw [line width=\edgewidth,-,red1] (x1) -- (u1);
                \draw [line width=\edgewidth,-,red1] (x1) -- (v1);
                \draw [line width=\edgewidth,-,red1] (y1) -- (u1);
                \draw [line width=\edgewidth,-,red1] (y1) -- (v1);
                \draw [line width=\edgewidth,-,red1] (x2) -- (u2);
                \draw [line width=\edgewidth,-,red1] (x2) -- (v2);
                \draw [line width=\edgewidth,-,red1] (y2) -- (u2);
                \draw [line width=\edgewidth,-,red1] (y2) -- (v2);
                \draw [line width=\edgewidth,dashed,red1] (v1) -- (z);
                 \draw [line width=\edgewidth,dashed,red1] (u2) -- (z);

			\end{tikzpicture}
			\caption{A path between $x_1$ and $y_2$ yields a subdivision of \cref{fig:graph4}.}
			\label{fig:disCycCase1_5}
			
		\end{subfigure}
		\hfill
            \begin{subfigure}[t]{0.28\textwidth}
			\centering
			\begin{tikzpicture}
               \def\vtxSize{0.5cm}
                \def\width{1.25cm}
			\def\height{1.25cm}
                \def\edgewidth{1.25pt}
                \def\r{.5cm}
			\node[draw,circle,minimum size=\vtxSize,inner sep=0pt] (x1) at (-       \width,\height) {$x_1$};
                \node[draw,circle,minimum size=\vtxSize,inner sep=0pt] (y1) at (-\width,-\height) {$y_1$};
                \node[draw,circle,minimum size=\vtxSize,inner sep=0pt] (x2) at  (\width, \height){$x_2$};
                \node[draw,circle,minimum size=\vtxSize,inner sep=0pt] (y2) at  (\width, -\height){$y_2$};
                \node[draw,circle,minimum size=\vtxSize,inner sep=0pt] (u1) at  (-\width-1.75*\r, 0){$u_1$};
                \node[draw,circle,minimum size=\vtxSize,inner sep=0pt] (v1) at  (-\width+1.75*\r, 0){$v_1$};
                \node[draw,circle,minimum size=\vtxSize,inner sep=0pt] (u2) at  (\width-\r, 0){$u_2$};
                \node[draw,circle,minimum size=\vtxSize,inner sep=0pt] (v2) at  (\width+\r, 0){$v_2$};
			 \node[draw,circle,minimum size=.5*\vtxSize,inner sep=1pt,fill, color=black] (z) at  ($(u1)!.33!(v1)$){};
            \node[draw,circle,minimum size=.5*\vtxSize,inner sep=1pt,fill, color=black] (z') at  ($(u1)!.66!(v1)$){};

			\draw [line width=\edgewidth,-,red1] (x1) -- (x2);
                \draw [line width=\edgewidth,-,red1] (y1) -- (y2);
                \draw [line width=\edgewidth,-,red1] (x1) -- (u1);
                \draw [line width=\edgewidth,-,red1] (x1) -- (v1);
                \draw [line width=\edgewidth,-,red1] (y1) -- (u1);
                \draw [line width=\edgewidth,-,red1] (y1) -- (v1);
                \draw [line width=\edgewidth,-,red1] (x2) -- (u2);
                \draw [line width=\edgewidth,-,red1] (x2) -- (v2);
                \draw [line width=\edgewidth,-,red1] (y2) -- (u2);
                \draw [line width=\edgewidth,-,red1] (y2) -- (v2);
                \draw [line width=\edgewidth,dashed,red1] (u1) -- (z);
                 \draw [line width=\edgewidth,dashed,red1] (z) -- (z');
                 \draw [line width=\edgewidth,dashed,red1] (v1) -- (z');

			\end{tikzpicture}
			\caption{A path between $u_1$ and $v_1$ gives a \cref{fig:graph16} subdivision.}
			\label{fig:disCycCase1_7}
			
		\end{subfigure}
		\caption{}
	\end{figure}
 
\noindent
{\it Case 2:} Here we view $C_2$ as the union of an $(x_2y_2)$-path $P$, $|P|\geq 2$, 
and the edge $x_2y_2$, which, by assumption is non-compliant. Let us consider first the case 
where $x_2,y_2$ separate $C_2$ from $H - C_2$. By \cref{lem:compliantOuterplanar}, 
$C_2$ is contained in a $2$-connected non-outerplanar subgraph which is disjoint from $H-C_2$. By \cref{lem:nonOuterStucture} the vertices 
$x_2$ and $y_2$ are contained a subdivision $F$, of either $K_4$ or $K_{2,3}$, disjoint from $H$. If $x_2$ and $y_2$ are not neighbors in $F$ then they are contained in some cycle of $F$, and we are back to Case 1 which we have already settled. 
Similarly, if $x_2$, $y_2$ are adjacent in $F$, then it is
easily seen that we can find a cycle $C'\subset F - x_2y_2$ and two disjoint paths from $x_2,y_2$ to $C'$, so we can again reduce the situations to Case 1.\\
So we may assume that $x_2,y_2$ do not disconnect $C_2$ from the rest of $H$. In particular, there exists a path $Q$ (possibly an edge) from an internal vertex, $u$, of $P$ to a vertex $v$ in $H-P$ which is otherwise disjoint from $H$. Up to symmetries there are three possibilities to consider: $Q$ connects between $u$ and
i) an internal vertex of $R_1$. ii) $x_1$, 
iii) $C_1-\{x_1,y_1\}$. In i), the situation reduces to Case 1, \cref{fig:disCycCase2_1} , which is already settled. In ii), G contains a subdivision of \cref{fig:graph2}, see \cref{fig:disCycCase2_2}. In iii), $G$ contains a subdivision of \cref{fig:graph5}.
\begin{figure}
		\centering
		\begin{subfigure}[t]{0.28\textwidth}
			\centering
			\begin{tikzpicture}
               \def\vtxSize{0.5cm}
                \def\width{1.25cm}
			\def\height{1.25cm}
                \def\edgewidth{1.25pt}
                \def\r{.75cm}
			\node[draw,circle,minimum size=\vtxSize,inner sep=0pt] (x1) at (-\width,\height) {$x_1$};
                \node[draw,circle,minimum size=\vtxSize,inner sep=0pt] (y1) at (-\width,-\height) {$y_1$};
                \node[draw,circle,minimum size=\vtxSize,inner sep=0pt] (x2) at  (\width, \height){$x_2$};
                \node[draw,circle,minimum size=\vtxSize,inner sep=0pt] (y2) at  (\width, -\height){$y_2$};
                \node[draw,circle,minimum size=\vtxSize,inner sep=0pt] (u1) at  (-\width-\r, 0){$u_1$};
                \node[draw,circle,minimum size=\vtxSize,inner sep=0pt] (v1) at  (-\width+\r, 0){$v_1$};
                \node[draw,circle,minimum size=\vtxSize,inner sep=0pt] (u2) at  (\width-\r, 0){$u_2$};
                \node[draw,circle,minimum size=\vtxSize,inner sep=0pt] (v2) at  (\width+\r, 0){$v_2$};
                \node[draw,circle,minimum size=.5*\vtxSize,inner sep=1pt,fill, color=black] (z) at  ($(u1)!.5!(v1)+ (0,\r/2)$){};
                \node[draw,circle,minimum size=.5*\vtxSize,inner sep=1pt,fill, color=black] (z') at  ($(u1)!.5!(v1)+ (0,-\r/2)$){};

			\draw [line width=\edgewidth,-,red1] (x1) -- (x2);
                \draw [line width=\edgewidth,-,red1] (y1) -- (y2);
                \draw [line width=\edgewidth,-,red1] (x1) -- (u1);
                \draw [line width=\edgewidth,-,red1] (x1) -- (v1);
                \draw [line width=\edgewidth,-,red1] (y1) -- (u1);
                \draw [line width=\edgewidth,-,red1] (y1) -- (v1);
                \draw [line width=\edgewidth,-] (x2) -- (u2);
                \draw [line width=\edgewidth,-,red1] (x2) -- (v2);
                \draw [line width=\edgewidth,-] (y2) -- (u2);
                \draw [line width=\edgewidth,-,red1] (y2) -- (v2);
                \draw [line width=\edgewidth,-,red1] (u1) -- (z);
                \draw [line width=\edgewidth,-,red1] (v1) -- (z);
                \draw [line width=\edgewidth,-,red1] (u1) -- (z');
                \draw [line width=\edgewidth,-,red1] (v1) -- (z');

			\end{tikzpicture}
			\caption{A subdivision of \cref{fig:graph10}.}
			\label{fig:disCycCase1_3}
			
		\end{subfigure}
            \hfill 
            \begin{subfigure}[t]{0.28\textwidth}
			\centering
			\begin{tikzpicture}
			\def\vtxSize{0.5cm}
                \def\width{1.25cm}
			\def\height{1.25cm}
                \def\edgewidth{1.5pt}
                \def\r{\width/2}
			\node[draw,circle,minimum size=\vtxSize,inner sep=0pt] (x1) at (-\width,\height) {$x_1$};
                \node[draw,circle,minimum size=\vtxSize,inner sep=0pt] (y1) at (-\width,-\height) {$y_1$};
                \node[draw,circle,minimum size=\vtxSize,inner sep=0pt] (x2) at  (\width, \height){$x_2$};
                \node[draw,circle,minimum size=\vtxSize,inner sep=0pt] (y2) at  (\width, -\height){$y_2$};
			\node[draw,circle,minimum size=.5*\vtxSize,inner sep=1pt,fill, color=black] (z) at  ($(x1)!.5!(x2)$){};
                \node[draw,circle,minimum size=\vtxSize,inner sep=0pt] (z') at  ($(x2)!.5!(y2) - (\r,0)$){$u$};

			\draw [line width=\edgewidth,dotted,red1] (x1) -- (z) ;
                \draw [line width=\edgewidth,dotted,red1] (x2) -- (z) ;
                \draw [line width=\edgewidth,dotted,red1] (y1) -- (y2);
                \draw [line width=\edgewidth,dotted,red1] (x1) to[out=-135, in=135] (y1);
                \draw [line width=\edgewidth,dotted,red1] (x1) to[out=-45, in=45] (y1);
                \draw [line width=\edgewidth,dotted] (x2) to[out=-135, in=90] (z');
                \draw [line width=\edgewidth,dotted,red1] (z') to[out=-90, in=135] (y2);
                \draw [line width=\edgewidth,-,red1] (x2) to[out=-45, in=45]   (y2);
                \draw [line width=\edgewidth,-,red1] (z) to (z');
      

			\end{tikzpicture}
			\caption{A path from $Q$ to $R_1$ recovers the previous case.}
			\label{fig:disCycCase2_1}
			
		\end{subfigure}
            \hfill 
            \begin{subfigure}[t]{0.28\textwidth}
			\centering
			\begin{tikzpicture}
			\def\vtxSize{0.5cm}
                \def\width{1.25cm}
			\def\height{1.25cm}
                \def\edgewidth{1.5pt}
                \def\r{\width/2}
			\node[draw,circle,minimum size=\vtxSize,inner sep=0pt] (x1) at (-\width,\height) {$x_1$};
                \node[draw,circle,minimum size=\vtxSize,inner sep=0pt] (y1) at (-\width,-\height) {$y_1$};
                \node[draw,circle,minimum size=\vtxSize,inner sep=0pt] (x2) at  (\width, \height){$x_2$};
                \node[draw,circle,minimum size=\vtxSize,inner sep=0pt] (y2) at  (\width, -\height){$y_2$};

                \node[draw,circle,minimum size=\vtxSize,inner sep=0pt] (z') at  ($(x2)!.5!(y2) - (\r,0)$){$u$};

			\draw [line width=\edgewidth,dotted,red1] (x1) -- (x2) ;
                \draw [line width=\edgewidth,dotted,red1] (y1) -- (y2);
                \draw [line width=\edgewidth,dotted,red1] (x1) to[out=-135, in=135] (y1);
                \draw [line width=\edgewidth,dotted,red1] (x1) to[out=-45, in=45] (y1);
                \draw [line width=\edgewidth,dotted] (x2) to[out=-135, in=90] (z');
                \draw [line width=\edgewidth,dotted,red1] (z') to[out=-90, in=135] (y2);
                \draw [line width=\edgewidth,-,red1] (x2) to[out=-45, in=45]   (y2);
                \draw [line width=\edgewidth,-,red1] (x1) to (z');
                \node[draw,circle,minimum size=.5*\vtxSize,inner sep=1pt,fill, color=black] (v) at  ($(x1)!.5!(y1) - (\r,0)$){};
                \node[draw,circle,minimum size=.5*\vtxSize,inner sep=1pt,fill, color=black] (v') at  ($(x1)!.5!(y1) + (\r,0)$){};

			\end{tikzpicture}
			\caption{An path from $P$ to $x_1$ gives a subdivision of \cref{fig:graph2}.}
			\label{fig:disCycCase2_2}
			
		\end{subfigure}
		\caption{}
	\end{figure}

\noindent
{\it Case 3:} Now each $C_i$ is the union of the edge $x_iy_i$ and an $(x_iy_i)$-path 
$P_i$ of length at least $2$. First assume that $x_2,y_w$ separate $C_2$ from $C_1$. As argued above, 
by \cref{lem:compliantOuterplanar} and \cref{lem:nonOuterStucture} the vertices $x_2$ and $y_2$ are contained in a subdivision $F$ of $K_{2,3}$ or $K_4$ which is disjoint from $C_1$. If $x_2$ and $y_2$ are not neighbors in $F$ then we can take $C_1$ and any cycle containing $x_2$ and $y_2$ and apply case 2. Similarly, if $x_2y_2$ is an edge in $F$ then we can take $C_1$ and a cycle in $F-x_2y_2$ and again apply case 2.\\
So we may assume that $x_2,y_2$  do not separate $C_2$ from $C_1$. This means that there is a path $Q$ (possibly an edge) connecting a vertex $u$ in $C_2-\{x_2,y_2\}$, i.e. an internal vertex of $P_2$, to a vertex in $H-C_2$. First assume that $Q$ is a path from $u$ to an internal vertex of $R_1$ (or $R_2$). 
(Recall that $R_1$ and $R_2$ are disjoint paths between $C_1$ and $C_2$, 
\cref{fig:disCycCase1_1}.)  
This creates the same situation as in Case 2, see e.g. \cref{fig:disCycCase2_1}.

So we may assume that $Q$ is a path from $u$ to $C_1$. Next suppose that $Q$ is a path connecting $u$ to an internal vertex of $P_1$. This gives us a
subdivision of the prism, \cref{fig:disCycCase3_1}, and as shown in \cite{CL},
a $2$-connected graph on $7$ vertices or more which contains a proper subdivision of the prism is not metrizable. 
We can therefore assume that the path $Q$ connects $u$ to the set $\{x_1,y_1\}$. The above discussion
allows us to assume as well that $x_1,y_1$ do not separate $C_1$ from $C_2$, and that there is a path $Q'$
from a vertex $u'\in C_1-\{x_1,y_1\}$ to the set $\{x_2,y_2\}$, which is internally disjoint from $H$. 
There are now up to symmetries 2 subcases to consider: i) $Q$ is a $(ux_1)$-path and $Q'$ is a $(u'x_2)$-path, ii) $Q$ is a $(ux_1)$-path and $Q'$ is $(u'y_2)$ path. In subcase i), $G$ contains a 
subdivision of the prism and has at least $7$ verices, \cref{fig:disCycCase3_2}, and as 
noted above, it is consequently non-metrizable. In subcase ii), notice that this graph contains a $K_{2,4}'$ subdivision, \cref{fig:disCycCase3_3}, which is not metrizable by \cref{lem:K24}.

\end{proof}
\begin{figure}
		\centering
		
		\hfill
            \begin{subfigure}[t]{0.28\textwidth}
			\centering
			\begin{tikzpicture}
			\def\vtxSize{0.5cm}
                \def\width{1.25cm}
			\def\height{1.25cm}
                \def\edgewidth{1.5pt}
                \def\r{\width/2}
			\node[draw,circle,minimum size=\vtxSize,inner sep=0pt] (x1) at (-\width,\height) {$x_1$};
                \node[draw,circle,minimum size=\vtxSize,inner sep=0pt] (y1) at (-\width,-\height) {$y_1$};
                \node[draw,circle,minimum size=\vtxSize,inner sep=0pt] (x2) at  (\width, \height){$x_2$};
                \node[draw,circle,minimum size=\vtxSize,inner sep=0pt] (y2) at  (\width, -\height){$y_2$};
		  \node[draw,circle,minimum size=\vtxSize,inner sep=0pt] (u1) at  ($(x1)!.5!(y1) + (\r,0)$){$u'$};
                \node[draw,circle,minimum size=\vtxSize,inner sep=0pt] (u2) at  ($(x2)!.5!(y2) - (\r,0)$){$u$};

			\draw [line width=\edgewidth,dotted,red1] (x1) -- (x2) ;
                \draw [line width=\edgewidth,dotted,red1] (y1) -- (y2);
                \draw [line width=\edgewidth,dotted,red1] (x1) to[out=-45, in=90] (u1);
                \draw [line width=\edgewidth,dotted,red1] (u1) to[out=-90, in=45] (y1);
                \draw [line width=\edgewidth,red1] (x1) to[out=-135, in=135] (y1);
                \draw [line width=\edgewidth,dotted,red1] (x2) to[out=-135, in=90] (u2);
                \draw [line width=\edgewidth,dotted,red1] (u2) to[out=-90, in=135] (y2);
                \draw [line width=\edgewidth,-,red1] (x2) to[out=-45, in=45]   (y2);
                \draw [line width=\edgewidth,-,red1] (u1) -- (u2);

			\end{tikzpicture}
			\caption{A path between $u$ and $u'$ gives a subdivision of the prism.}
			\label{fig:disCycCase3_1}
			
		\end{subfigure}
            \hfill
		\begin{subfigure}[t]{0.28\textwidth}
			\centering
			\begin{tikzpicture}
			\def\vtxSize{0.5cm}
                \def\width{1.25cm}
			\def\height{1.25cm}
                \def\edgewidth{1.5pt}
                \def\r{\width/2}
			\node[draw,circle,minimum size=\vtxSize,inner sep=0pt] (x1) at (-\width,\height) {$x_1$};
                \node[draw,circle,minimum size=\vtxSize,inner sep=0pt] (y1) at (-\width,-\height) {$y_1$};
                \node[draw,circle,minimum size=\vtxSize,inner sep=0pt] (x2) at  (\width, \height){$x_2$};
                \node[draw,circle,minimum size=\vtxSize,inner sep=0pt] (y2) at  (\width, -\height){$y_2$};
		  \node[draw,circle,minimum size=\vtxSize,inner sep=0pt] (u1) at  ($(x1)!.5!(y1) + (\r,0)$){$u'$};
                \node[draw,circle,minimum size=\vtxSize,inner sep=0pt] (u2) at  ($(x2)!.5!(y2) - (\r,0)$){$u$};

			\draw [line width=\edgewidth,dotted] (x1) -- (x2) ;
                \draw [line width=\edgewidth,dotted,red1] (y1) -- (y2);
                \draw [line width=\edgewidth,dotted,red1] (x1) to[out=-45, in=90] (u1);
                \draw [line width=\edgewidth,dotted,red1] (u1) to[out=-90, in=45] (y1);
                \draw [line width=\edgewidth, red1] (x1) to[out=-135, in=135] (y1);
                \draw [line width=\edgewidth,dotted,red1] (x2) to[out=-135, in=90] (u2);
                \draw [line width=\edgewidth,dotted,red1] (u2) to[out=-90, in=135] (y2);
                \draw [line width=\edgewidth,-,red1] (x2) to[out=-45, in=45]   (y2);
                \draw [line width=\edgewidth,-,red1, dotted] (x1) to (u2);
                \draw [line width=\edgewidth,-,red1, dotted] (x2) to (u1);

			\end{tikzpicture}
			\caption{Paths from $u'$ to $x_1$ and $u$ to $x_2$ yield a subdivision of the prism.}
			\label{fig:disCycCase3_2}
			
		\end{subfigure}
            \hfill 
            \begin{subfigure}[t]{0.28\textwidth}
			\centering
			\begin{tikzpicture}
			\def\vtxSize{0.5cm}
                \def\width{1.25cm}
			\def\height{1.25cm}
                \def\edgewidth{1.5pt}
                \def\r{\width/2}
			\node[draw,circle,minimum size=\vtxSize,inner sep=0pt] (x1) at (-\width,\height) {$x_1$};
                \node[draw,circle,minimum size=\vtxSize,inner sep=0pt] (y1) at (-\width,-\height) {$y_1$};
                \node[draw,circle,minimum size=\vtxSize,inner sep=0pt] (x2) at  (\width, \height){$x_2$};
                \node[draw,circle,minimum size=\vtxSize,inner sep=0pt] (y2) at  (\width, -\height){$y_2$};
		  \node[draw,circle,minimum size=\vtxSize,inner sep=0pt] (u1) at  ($(x1)!.5!(y1) + (\r,0)$){$u'$};
                \node[draw,circle,minimum size=\vtxSize,inner sep=0pt] (u2) at  ($(x2)!.5!(y2) - (\r,0)$){$u$};

			\draw [line width=\edgewidth,dotted,red1] (x1) -- (x2) ;
                \draw [line width=\edgewidth,dotted,red1] (y1) -- (y2);
                \draw [line width=\edgewidth,dotted,red1] (x1) to[out=-45, in=90] (u1);
                \draw [line width=\edgewidth,dotted] (u1) to[out=-90, in=45] (y1);
                \draw [line width=\edgewidth, red1] (x1) to[out=-135, in=135] (y1);
                \draw [line width=\edgewidth,dotted,red1] (x2) to[out=-135, in=90] (u2);
                \draw [line width=\edgewidth,dotted,red1] (u2) to[out=-90, in=135] (y2);
                \draw [line width=\edgewidth,-,red1] (x2) to[out=-45, in=45]   (y2);
                \draw [line width=\edgewidth,-,red1, dotted] (x1) to (u2);
                \draw [line width=\edgewidth,-,red1, dotted] (y2) to (u1);

			\end{tikzpicture}
			\caption{A subdivision of $K_{2,4}'$, with $x_1, y_2$ on the $2$-part.}
			\label{fig:disCycCase3_3}
			
		\end{subfigure}
		\caption{}
	\end{figure}

Our proof of \cref{thm:mainThm} uses the following result of Lov\'asz, \cite{L}.
Consider the graphs obtained by adding $1$, $2$ and $3$ edges to the side of $3$ vertices in $K_{3,n}$. Call them
$K_{3,n}'$, $K_{3,n}''$ and $K_{3,n}'''$, respectively.
\begin{lemma}\label{lem:Lovasz}
If a graph $G$ contains no disjoint cycles and satisfies $\delta(G)\geq 3$, then it is either $K_5$, a wheel, $K_{3,n}$, $K_{3,n}'$, $K_{3,n}''$ or $K_{3,n}'''$.
\end{lemma}

\begin{proof}{[\Cref{thm:mainThm}]} 
Let $G$ be a $2$-connected graph with no compliant edges and at least $11$ vertices.
We may assume that $G$ does not contain disjoint cycles,
for otherwise it is not metrizable by \cref{lem:disjointCycles}.\\ 
In order to use \cref{lem:Lovasz}, we need to eliminate all vertices of degree $2$ in $G$, 
since the lemma assumes $\delta(G)\geq 3$. If we suppress all vertices of degree $2$ in $G$,
the resulting graph $\tilde{G}$ indeed satisfies $\delta(\tilde{G})\geq 3$, though
it need not be a {\em simple} graph. It is not difficult to see that $\tilde{G}$ is
simple if and only if $G$ does not contain any parallel flat paths. In this view let us consider
what happens if $G$ contains parallel flat paths. Namely, there are at least two flat $(uv)$-paths in $G$ between a pair of
vertices $u,v\in G$. We show that there cannot be an additional vertex pair $\{x,y\} \neq \{u,v\}$ that is connected by parallel flat paths. Now $\{x,y\}$, $\{u,v\}$ cannot be disjoint, or else
these four flat paths form two disjoint cycles. So suppose that there are parallel flat paths between $x,y$ and $x,v$. Since $G$ is $2$-connected, there must be another path from $v$ to $y$ disjoint from $x$. But this forms a subdivision of \cref{fig:graph2}.\\
So say that $x,y$ is the (unique) pair of vertices between which there are parallel flat paths. If $x$ and $y$ are the only 
branch vertices of $G$
then $G$ is a subdivision of $K_{2,n}$, $n\geq 3$. But as shown in \cite{CL}, for $n\geq 4$ the graph
$K_{2,n}$ is metrizable, whereas every proper subdivision of $K_{2,n}$ is not metrizable,
since it contains a subdivision of \cref{fig:graph1}. In the remaining case $n=3$ and $G$ is a subdivided $K_{2,3}$.\\ 
So we may assume that $G$ has another branch vertex $z\neq x, y$. Since $G$ is $2$-connected there is an $(xy)$-path in $G$ containing $z$. Let $P$ be such a path of 
largest length. Note that every
cycle $Z$ in $G$ must contain either $x$ or $y$. Otherwise, $Z$
along with the parallel $xy$ paths, form two disjoint cycles. It follows that if 
and $Q$ is a $(uv)$-path internally disjoint from $P$ with $u,v\in P$,
then $\{u,v\}\cap\{x,y\}\neq \emptyset$. Otherwise, $Q$ along with the $(uv)$-subpath of $P$ forms a a cycle containing neither $x$ nor $y$. \\
We first suppose that there exists a vertex $u\in P - \{x,y\}$ with a neighbor $v\notin P$. Since $G$ is $2$-connected there are two paths from $v$ to $P$ disjoint except at $v$. At least one of these paths does not contain $u$, and gluing this path with the edge $uv$ we obtain a path $Q$ from $u$ to another vertex in $P$. As was observed above, this other endpoint must be either $x$ or $y$. Say $Q$ is a $(ux)$-path, and note that $u$ and $x$ are not neighbors in $P$ for otherwise we can use $Q$ to replace $P$ by a longer $(xy)$-path, contradicting the maximality of $P$'s length. Then the graph $H\cup P \cup Q$ contains a subdivision of \cref{fig:graph2}, implying $G$ is not metrizable.\\
In the remaining case every internal vertex of $P$ has all its neighbors in $P$. Given an internal vertex of $u$ in $P$, let us call $w$ an additional neighbor of $u$ if $uw$ is an edge in $G$ and but not an edge in $P$. Note by assumption, all additional neighbors are all also vertices in $P$. Moreover, by our above observations if $w$ is an additional neighbor of $u$ then $w\in \{x,y\}$. (Recall that otherwise we would obtain disjoint cycles in $G$.) Recall that $P$ has at least one vertex $z$, of degree at least $3$. So $z$ has at least one additional neighbor, say $x$. Let $u$ be the vertex in $P$ such that $x$ is an additional neighbor of $u$ and the distance between $u$ and $x$ along $P$ is minimized. Since $ux$ is not a compliant edge, the $(ux)$-subpath of $P$ is not flat. In particular, there must be a vertex, $v$, of degree $3$ along this subpath. This vertex $v$ has $x$ or $y$ as an additional neighbor. Also, since $u$ was chosen to minimizes the distance to $x$, $v$ must have $y$ as an additional neighbor. Along with the flat $(xy)$-paths we obtained a subdivision of $K_{2,4}'$. It follows that $G$ is not metrizable by \cref{lem:K24}.\\
In the eventual case $G$ has no parallel flat paths, so
we may suppress all its degree-$2$ vertices to obtain a simple graph $\tilde{G}$ such that $\delta(\tilde{G})\geq 3$. 
By \cref{lem:Lovasz} $\tilde{G}$ is either $K_5$, $W_n$, $K_{3,n}$, $K_{3,n}'$, $K_{3,n}''$ or $K_{3,n}'''$. As shown in \cite{CL}, every $2$-connected non-planar graph with at least $8$ vertices is non-metrizable. 
Since $G$ has at least $11$ vertices this excludes $K_5$, $K_{3,n}$, $K_{3,n}'$, $K_{3,n}''$ or $K_{3,n}'''$ for $n\geq 3$.
To exclude $W_n$ for $n\geq 5$, recall
from \cite{CL} that a $2$-connected graph containing $W_5$ with at least $7$ vertices is non-metrizable. The
possibilities for $\tilde{G}$ that remain are precisely $K_4$, $W_4$ and $W_4'$. 

\end{proof}
Next we prove \cref{cor:outerplanar}.
\begin{proof}{[\Cref{cor:outerplanar}]} 
To prove the claim we use the following procedure: we first iteratively remove all compliant edges in $G$ to obtain a graph $G_t$ with no compliant edges and conclude that there exists a vertex $x$ such that $G_t-x$ is outerplanar. The last step is to recover the graph $G-x$ by iteratively adding back the compliant edges we removed while maintining outerplanarity. In this last step, we use the fact that an outerplanar graph remains outerplanar after adding a compliant edge, \cref{prop:compliantEdges} (2).\\
We first observe that every compliant edge $e$ in $G$, is also compliant in 
$G-x$, for every vertex
$x\in G$ of degree at least $3$ that is not incident with $e$. 
Start with $G_0=G$ and define the graphs
$G_0,G_1,\dots, G_t$, where $G_{i+1}$ is obtained from $G_{i}$ by deleting
some compliant edge $e_i$ of $G_i$.
This sequence ends with $G_t$ that has no compliant edges. 
Suppose that there exists a vertex $x\in G_t$ of degree at least $3$ such that $G_t-x$ is outerplanar. By our initial observation, if $e_i$ is not incident with $x$ then $e_i$ is a compliant edge in $G_i-x$. It follows that $G-x$ can be constructed from $G_t-x$ by iteratively adding compliant edge. By \cref{prop:compliantEdges}, 
$G$ is also outerplanar. \\
Why is there such a vertex $x$ in $G_t$? By
\cref{prop:compliantEdges}, $G_t$ is $2$-connected and so by \cref{thm:mainThm} it is either $K_{2,n}$ or a subdivision of $C_3$, $K_{2,3}$, $K_4$, $W_4$ or $W_4'$. Unless $G_t$ is a subdivision of $W_4'$ it is clear such a vertex exists and so it suffices to consider this case. But
$W_4'$ is the wheel $W_4$ plus an additional edge $e$, so $G_t =H\cup P$ where $H$ is the subdivision of $W_4$ and $P$ is the subdivision of $e$. Notice that if $|P|\geq 2$ then $G_t$ contains a subdivision of $K_{2,4}'$, and by \cref{lem:K24}, $G_t$ is not metrizable.  It follows that $P$ is just an edge and $G_n$ can be made outerplanar by removing the single vertex of degree $4$ in $H$.
\end{proof}
\section{Open Questions}
While the structure of metrizable graphs is now much better understood,
many questions remain open. For instance, \cref{thm:mainThm} shows that, up to compliant edges, large metrizable graphs are either $K_{2,n}$ or subdivisions of $C_3,K_4,K_{2,3},W_4$ or $W_4'$, but the converse is far from true. Indeed, many of the non-metrizable graphs in \cref{fig:zoo} are subdivisions of these graphs. Our ignorance is perhaps best illustrated by $\Theta$-graphs. 
Recall that $\Theta_{a,b,c}$ consists of two vertices $u$ and $v$ and three 
internally disjoint $(uv)$-paths of length $a$, $b$ and $c$. As mentioned already in 
\cite{CL}, we still do not know in full which $\Theta$-graphs are metrizable. 
\begin{open}
Which subdivisions of $K_4,K_{2,3},W_4$ and $W_4'$ are metrizable and which are not?
\end{open}
In \cite{CL} it was shown that graph metrizability can be decided in polynomial time. The proof of this relied heavily on the graph minor theory of Robertson and Seymour, \cite{RS}, and therefore a practical metrizability algorithm still remains elusive. On the other hand,
our results here show that metrizable graphs have a very simple structure. In this view
we could hope to use these insights and obtain a practically efficient metrizability algorithm. Unfortunately, even such simple graphs have exponentially many consistent path systems and so an exhuastive checking of all consistent path systems is still inefficient. 
\begin{open}
Find a practically efficient algorithm for graph metrizability.
\end{open}
It is suggestive to go beyond decision problems into the realm
of {\em promise problems}, and ask for example: 
\begin{open}
Given a non-metrizable graph, can we efficiently
find a consistent non-metrizable path system?
\end{open}
As mentioned in section 2, given non-metric path system we can find a ``hand-checkable'' certificate of non-metrizability using LP-duality.
On the other hand, it is not clear that such a certificate exists for metrizable graphs.
\begin{open}
Do there exist ``humanly verifiable'' certificates of metrizability?
\end{open}
Our definitions do not posit that path systems use all the edges in the graph
under consideration. In this view we find the concept of neighborly metrizability 
particularly appealing. A {\em neighborly consistent path system} is a consistent path 
system where every edge is
the chosen path between its two vertices. We say a graph $G$ is {\em neighborly metrizable}
if every consistent neighborly path system in $G$ is induced by a metric. 
(For example, $K_n$ is always neighborly metrizable but never metrizable for $n\geq 7$.) 
\begin{open}
    Characterize the family of graph which are neighborly metrizable.
\end{open}
Unlike metrizable graphs, neighborly metrizable graphs are not closed under topological minors. Therefore, different tools are necessary to tackle this problem.
\begin{figure}
		\centering
		\begin{subfigure}{0.175\textwidth}
			\centering
			\resizebox{\textwidth}{!}{
				\begin{tikzpicture}
				
				\node[draw,circle,fill] (2) at (1*360/5 +90: 5cm) {$2$};
				\node[draw,circle,fill]  (3) at (2*360/5 +90: 5cm) {$3$};
				\node[draw,circle,fill] (4) at (3*360/5 +90: 5cm) {$4$};
				\node[draw,circle,fill] (5) at (4*360/5 +90: 5cm) {$5$};
				\node[draw,circle,fill] (1) at ($(2)!0.5!(5)$) {$1$};
				\node[draw,circle,fill] (6) at ($(2)!0.5!(5) + (0,-3)$) {$6$};
				\node[draw,circle,fill] (7) at  ($(2)!0.5!(5) + (0,3)$) {$7$};

				\draw [line width=3pt,-] (1) -- (2);
				\draw [line width=3pt,-] (2) -- (3);
				\draw [line width=3pt,-] (3) -- (4);
				\draw [line width=3pt,-] (4) -- (5);
				\draw [line width=3pt,-] (5) -- (1);
				\draw [line width=3pt,-] (2) -- (6);
				\draw [line width=3pt,-] (5) -- (6);
				\draw [line width=3pt,-] (2) -- (7);
				\draw [line width=3pt,-] (5) -- (7);
				
				\end{tikzpicture}
			}
			\caption{Graph 1}
			\label[graph]{fig:graph1}
		\end{subfigure}
		\hfill
		\begin{subfigure}{0.175\textwidth}
			\centering
			\resizebox{\textwidth}{!}{
				\begin{tikzpicture}

				\node[draw,circle,fill] (1) at (0*360/5 +90: 5cm) {$1$};
				\node[draw,circle,fill] (2) at (1*360/5 +90: 5cm) {$2$};
				\node[draw,circle,fill] (3) at (2*360/5 +90: 5cm) {$3$};
				\node[draw,circle,fill] (4) at (3*360/5 +90: 5cm) {$4$};
				\node[draw,circle,fill] (5) at (4*360/5 +90: 5cm) {$5$};
				
				\node[draw,circle,fill] (6) at ($(1)!0.5!(3)$) {$6$};
				\node[draw,circle,fill] (7) at  ($(1)!0.5!(4)$) {$7$};

				\draw [line width=3pt,-] (1) -- (2);
				\draw [line width=3pt,-] (2) -- (3);
				\draw [line width=3pt,-] (3) -- (4);
				\draw [line width=3pt,-] (4) -- (5);
				\draw [line width=3pt,-] (5) -- (1);
				\draw [line width=3pt,-] (1) -- (6);
				\draw [line width=3pt,-] (3) -- (6);
				\draw [line width=3pt,-] (1) -- (7);
				\draw [line width=3pt,-] (4) -- (7);
				
				\end{tikzpicture}
			}
			\caption{Graph 2}
			\label[graph]{fig:graph2}
		\end{subfigure}
		\hfill
		\begin{subfigure}{0.175\textwidth}
			\centering
			\resizebox{\textwidth}{!}{
				\begin{tikzpicture}

				\node[draw,circle,fill] (1) at (0*360/5 +90: 5cm) {$1$};
				\node[draw,circle,fill]  (2) at (1*360/5 +90: 5cm) {$2$};
				\node[draw,circle,fill] (3) at (2*360/5 +90: 5cm) {$3$};
				\node[draw,circle,fill]  (4) at (3*360/5 +90: 5cm) {$4$};
				\node[draw,circle,fill] (5) at (4*360/5 +90: 5cm) {$5$};
				
				\node[draw,circle,fill] (6) at ($(1)!0.5!(3)$) {$6$};
				\node[draw,circle,fill] (7) at  ($(1)!0.5!(4)$) {$7$};

				\draw [line width=3pt,-] (1) -- (2);
				\draw [line width=3pt,-] (2) -- (3);
				\draw [line width=3pt,-] (3) -- (4);
				\draw [line width=3pt,-] (4) -- (5);
				\draw [line width=3pt,-] (5) -- (1);
				\draw [line width=3pt,-] (6) -- (7);
				\draw [line width=3pt,-] (2) -- (6);
				\draw [line width=3pt,-] (5) -- (7);
				\draw [line width=3pt,-] (3) -- (7);
				\draw [line width=3pt,-] (4) -- (6);
				
				\end{tikzpicture}
			}
			\caption{Graph 3}
			\label[graph]{fig:graph3}
		\end{subfigure}
            \\
		\begin{subfigure}{0.175\textwidth}
			\centering
			\resizebox{\textwidth}{!}{
				\begin{tikzpicture}

				\node[draw,circle,fill] (1) at (0*360/5 +90: 5cm) {$1$};
				\node[draw,circle,fill]  (2) at (1*360/5 +90: 5cm) {$2$};
				\node[draw,circle,fill] (3) at (2*360/5 +90: 5cm) {$3$};
				\node[draw,circle,fill] (4) at (3*360/5 +90: 5cm) {$4$};
				\node[draw,circle,fill]  (5) at (4*360/5 +90: 5cm) {$5$};
				
				\node (m1) at ($(2)!.5! (3)$) {};
				\node (m2) at ($(4)!.5! (5)$) {};
				\node[draw,circle,fill] (6) at ($(m1)!.8!270:(2)$) {$6$};
				\node[draw,circle,fill] (7) at  ($(m2)!.8!90:(5)$) {$7$};

				\draw [line width=3pt,-] (1) -- (2);
				\draw [line width=3pt,-] (2) -- (3);
				\draw [line width=3pt,-] (3) -- (4);
				\draw [line width=3pt,-] (4) -- (5);
				\draw [line width=3pt,-] (5) -- (1);
				\draw [line width=3pt,-] (6) -- (7);
				\draw [line width=3pt,-] (2) -- (6);
				\draw [line width=3pt,-] (5) -- (7);
				\draw [line width=3pt,-] (3) -- (6);
				\draw [line width=3pt,-] (4) -- (7);

				\end{tikzpicture}
			}
			\caption{Graph 4}
			\label[graph]{fig:graph4}
		\end{subfigure}
            \hfill
		\begin{subfigure}{0.175\textwidth}
			\centering
			\resizebox{\textwidth}{!}{
				\begin{tikzpicture}
				
				\node[draw,circle,fill] (1) at (0*360/5 +90: 5cm) {$1$};
				\node[draw,circle,fill] (2) at (1*360/5 +90: 5cm) {$2$};
				\node[draw,circle,fill] (3) at (2*360/5 +90: 5cm) {$3$};
				\node[draw,circle,fill] (4) at (3*360/5 +90: 5cm) {$4$};
				\node[draw,circle,fill] (5) at (4*360/5 +90: 5cm) {$5$};
				\node[draw,circle,fill] (6) at ($(2)!.5!(5)+(0,-1)$) {$6$};
				\node[draw,circle,fill] (7) at ($(2)!.5!(5) + (0,-3)$) {$7$};
				
				\draw [line width=3pt,-] (1) -- (2);
				\draw [line width=3pt,-] (2) -- (3);
				\draw [line width=3pt,-] (3) -- (4);
				\draw [line width=3pt,-] (4) -- (5);
				\draw [line width=3pt,-] (5) -- (1);
				\draw [line width=3pt,-] (6) -- (7);
				\draw [line width=3pt,-] (2) -- (6);
				\draw [line width=3pt,-] (5) -- (6);
				\draw [line width=3pt,-] (3) -- (7);
				\draw [line width=3pt,-] (4) -- (7);
				
				\end{tikzpicture}
			}
			\caption{Graph 5}
			\label[graph]{fig:graph5}
		\end{subfigure}
		\hfill
		\begin{subfigure}{0.175\textwidth}
			\centering
			\resizebox{\textwidth}{!}{
				\begin{tikzpicture}
				
				\node[draw,circle,fill] (1) at (0*360/7 +90: 5cm) {$1$};
				\node[draw,circle,fill] (2) at (1*360/7 +90: 5cm) {$2$};
				\node[draw,circle,fill] (3) at (2*360/7 +90: 5cm) {$3$};
				\node[draw,circle,fill] (4) at (3*360/7 +90: 5cm) {$4$};
				\node[draw,circle,fill] (5) at (4*360/7 +90: 5cm) {$5$};
				\node[draw,circle,fill] (6) at (5*360/7 +90: 5cm) {$6$};
				\node[draw,circle,fill] (7) at (6*360/7 +90: 5cm) {$7$};
				\node[draw,circle,fill] (8) at (0,0) {$8$};
				
				\draw [line width=3pt,-] (1) -- (2);
				\draw [line width=3pt,-] (2) -- (3);
				\draw [line width=3pt,-] (3) -- (4);
				\draw [line width=3pt,-] (4) -- (5);
				\draw [line width=3pt,-] (5) -- (6);
				\draw [line width=3pt,-] (6) -- (7);
				\draw [line width=3pt,-] (7) -- (1);
				\draw [line width=3pt,-] (1) -- (8);
				\draw [line width=3pt,-] (3) -- (8);
				\draw [line width=3pt,-] (6) -- (8);
				
				\end{tikzpicture}
			}
			\caption{Graph 6}
			\label[graph]{fig:graph6}
		\end{subfigure}
            \\
		\begin{subfigure}{0.175\textwidth}
			\centering
			\resizebox{\textwidth}{!}{
				\begin{tikzpicture}
				
				\node[draw,circle,fill] (1) at (0*360/5 +90: 5cm) {$1$};
				\node[draw,circle,fill] (2) at (1*360/5 +90: 5cm) {$2$};
				\node[draw,circle,fill] (3) at (2*360/5 +90: 5cm){$3$};
				\node[draw,circle,fill] (4) at (3*360/5 +90: 5cm) {$4$};
				\node[draw,circle,fill] (5) at (4*360/5 +90: 5cm)  {$5$};
				\node[draw,circle,fill] (6) at (0,0) {$6$};
				\node[draw,circle,fill] (7) at ($(1)!.5!(6)$) {$7$};
				\node[draw,circle,fill] (8) at ($(3)!.5!(4)$) {$8$};
				
				\draw [line width=3pt,-] (1) -- (2);
				\draw [line width=3pt,-] (1) -- (5);
				\draw [line width=3pt,-] (1) -- (7);
				\draw [line width=3pt,-] (2) -- (3);
				\draw [line width=3pt,-] (3) -- (8);
				\draw [line width=3pt,-] (4) -- (5);
				\draw [line width=3pt,-] (4) -- (8);
				\draw [line width=3pt,-] (3) -- (6);
				\draw [line width=3pt,-] (4) -- (6);
				\draw [line width=3pt,-] (6) -- (7);
				
				\end{tikzpicture}
			}
			\caption{Graph 7}
			\label[graph]{fig:graph7}
		\end{subfigure}
		\hfill
		\begin{subfigure}{0.175\textwidth}
			\centering
			\resizebox{\textwidth}{!}{
				
				\begin{tikzpicture}
				
				\node[draw,circle,fill] (1) at (0*360/5 +90: 5cm) {$1$};
				\node[draw,circle,fill] (2) at (1*360/5 +90: 5cm) {$2$};
				\node[draw,circle,fill] (3) at (2*360/5 +90: 5cm){$3$};
				\node[draw,circle,fill] (4) at (3*360/5 +90: 5cm) {$4$};
				\node[draw,circle,fill] (5) at (4*360/5 +90: 5cm)  {$5$};
				\node[draw,circle,fill] (6) at (0,0) {$6$};
				\node[draw,circle,fill] (7) at ($(1)!.33!(6)$) {$7$};
				\node[draw,circle,fill] (8) at ($(1)!.66!(6)$) {$8$};
				
				\draw [line width=3pt,-] (1) -- (2);
				\draw [line width=3pt,-] (1) -- (5);
				\draw [line width=3pt,-] (1) -- (7);
				\draw [line width=3pt,-] (2) -- (3);
				\draw [line width=3pt,-] (3) -- (4);
				\draw [line width=3pt,-] (4) -- (5);
				\draw [line width=3pt,-] (6) -- (7);
				\draw [line width=3pt,-] (7) -- (8);
				\draw [line width=3pt,-] (3) -- (6);
				\draw [line width=3pt,-] (4) -- (6);
				
				\end{tikzpicture}
			}
			\caption{Graph 8}
			\label[graph]{fig:graph8}
		\end{subfigure}
            \hfill
		\begin{subfigure}{0.175\textwidth}
			\centering
			\resizebox{\textwidth}{!}{
				\begin{tikzpicture}

				\node[draw,circle,fill] (1) at (0*360/5 +90: 5cm) {$1$};
				\node[draw,circle,fill]  (2) at (1*360/5 +90: 5cm) {$2$};
				\node[draw,circle,fill] (3) at (2*360/5 +90: 5cm) {$3$};
				\node[draw,circle,fill] (4) at (3*360/5 +90: 5cm) {$4$};
				\node[draw,circle,fill]  (5) at (4*360/5 +90: 5cm) {$5$};
				\node[draw,circle,fill]  (6) at (0,0) {$6$};
				\node[draw,circle,fill]  (7) at ($(3)!0.5!(6)$) {$7$};
				\node[draw,circle,fill] (8) at ($(4)!0.5!(6)$) {$8$};

				\draw [line width=3pt,-] (1) -- (2);
				\draw [line width=3pt,-] (1) -- (5);
				\draw [line width=3pt,-] (1) -- (6);
				\draw [line width=3pt,-] (2) -- (3);
				\draw [line width=3pt,-] (3) -- (4);
				\draw [line width=3pt,-] (3) -- (7);
				\draw [line width=3pt,-] (4) -- (5);
				\draw [line width=3pt,-] (4) -- (8);
				\draw [line width=3pt,-] (6) -- (7);
				\draw [line width=3pt,-] (6) -- (8);

				\end{tikzpicture}
			}
			\caption{Graph 9}
			\label[graph]{fig:graph9}
		\end{subfigure}
		\\
            \begin{subfigure}{0.175\textwidth}
			\centering
			\resizebox{\textwidth}{!}{
				\begin{tikzpicture}
				
				\node[draw,circle,fill] (1) at (0*360/5 +90: 5cm) {$1$};
				\node[draw,circle,fill] (2) at (1*360/5 +90: 5cm) {$2$};
				\node[draw,circle,fill] (3) at (2*360/5 +90: 5cm) {$3$};
				\node[draw,circle,fill] (4) at (3*360/5 +90: 5cm) {$4$};
				\node[draw,circle,fill] (5) at (4*360/5 +90: 5cm) {$5$};
                    \node[draw,circle,fill] (6) at (0,-1.6) {$6$};
				\node[draw,circle,fill] (7) at ($(2)!.25!(5)$) {$7$};
				\node[draw,circle,fill] (8) at ($(2)!.75!(5)$) {$8$};
				
				\draw [line width=3pt,-] (1) -- (2);
                    \draw [line width=3pt,-] (1) -- (5);
                    \draw [line width=3pt,-] (1) -- (7);
                    \draw [line width=3pt,-] (1) -- (8);
                    \draw [line width=3pt,-] (2) -- (3);
                    \draw [line width=3pt,-] (2) -- (6);
                    \draw [line width=3pt,-] (3) -- (4);
                    \draw [line width=3pt,-] (4) -- (5);
                    \draw [line width=3pt,-] (5) -- (6);
                    \draw [line width=3pt,-] (6) -- (7);
                    \draw [line width=3pt,-] (6) -- (8);

				\end{tikzpicture}
			}
			\caption{Graph 10}
			\label[graph]{fig:graph10}
		\end{subfigure}
            \hfill
            \begin{subfigure}{0.175\textwidth}
			\centering
			\resizebox{\textwidth}{!}{
				\begin{tikzpicture}
				
				\node[draw,circle,fill] (1) at (-4.5,0) {$1$};
				\node[draw,circle,fill] (2) at (4.5,0) {$2$};
				\node[draw,circle,fill] (3) at ($(1)!.5!(2)$) {$3$};
                    \node[draw,circle,fill] (4) at ($(1)!.5!(2)+ (0,-2)$) {$4$};
                    \node[draw,circle,fill] (5) at ($(1)!.5!(2)+ (0,2)$) {$5$};
                    \node[draw,circle,fill] (6) at ($(1)!.5!(2)+ (0,6)$) {$6$};
                    \node[draw,circle,fill] (7) at ($(5)!.33!(6)$) {$7$};
                    \node[draw,circle,fill] (8) at ($(5)!.66!(6)$) {$8$};

				\draw [line width=3pt,-] (1) -- (3);
                    \draw [line width=3pt,-] (1) -- (4);
                    \draw [line width=3pt,-] (1) -- (5);
                    \draw [line width=3pt,-] (1) -- (6);
                    \draw [line width=3pt,-] (2) -- (3);
                    \draw [line width=3pt,-] (2) -- (4);
                    \draw [line width=3pt,-] (2) -- (5);
                    \draw [line width=3pt,-] (2) -- (6);
                    \draw [line width=3pt,-] (5) -- (7);
                    \draw [line width=3pt,-] (7) -- (8);
                    \draw [line width=3pt,-] (6) -- (8);

				\end{tikzpicture}
			}
			\caption{Graph 11}
			\label[graph]{fig:graph11}
		\end{subfigure}
            \hfill
            \begin{subfigure}{0.175\textwidth}
			\centering
			\resizebox{\textwidth}{!}{
				\begin{tikzpicture}
				
				\node[draw,circle,fill] (1) at (-4.5,0) {$1$};
				\node[draw,circle,fill] (2) at (4.5,0) {$2$};
                    \node[draw,circle,fill] (3) at ($(1)!.5!(2)+ (0,6)$) {$3$};
                    \node[draw,circle,fill] (6) at ($(1)!.5!(2)+ (0,-2)$) {$6$};
                    \node[draw,circle,fill] (4) at ($(3)!.4!(6)$) {$4$};
                    \node[draw,circle,fill] (5) at ($(3)!.8!(6)$) {$5$};
                    \node[draw,circle,fill] (7) at ($(3)!.2!(6)$) {$7$};
                    \node[draw,circle,fill] (8) at ($(3)!.6!(6)$) {$8$};

				\draw [line width=3pt,-] (1) -- (3);
                    \draw [line width=3pt,-] (1) -- (4);
                    \draw [line width=3pt,-] (1) -- (5);
                    \draw [line width=3pt,-] (1) -- (6);
                    \draw [line width=3pt,-] (2) -- (3);
                    \draw [line width=3pt,-] (2) -- (4);
                    \draw [line width=3pt,-] (2) -- (5);
                    \draw [line width=3pt,-] (2) -- (6);
                    \draw [line width=3pt,-] (3) -- (7);
                    \draw [line width=3pt,-] (4) -- (7);
                    \draw [line width=3pt,-] (4) -- (8);
                    \draw [line width=3pt,-] (5) -- (8);

				\end{tikzpicture}
			}
			\caption{Graph 12}
			\label[graph]{fig:graph12}
		\end{subfigure}
            \\
		\begin{subfigure}{0.175\textwidth}
			\centering
			\resizebox{\textwidth}{!}{
				\begin{tikzpicture}
				
				\node[draw,circle,fill] (1) at (0*360/6 +120: 5cm) {$1$};
				\node[draw,circle,fill] (2) at (1*360/6 +120: 5cm) {$2$};
				\node[draw,circle,fill] (3) at (2*360/6 +120: 5cm) {$3$};
				\node[draw,circle,fill] (5) at (3*360/6 +120: 5cm) {$5$};
				\node[draw,circle,fill] (6) at (4*360/6 +120: 5cm) {$6$};
				\node[draw,circle,fill] (7) at (5*360/6 +120: 5cm) {$7$};
				\node[draw,circle,fill] (4) at ($(3)!.5!(5)$) {$4$};
				\node[draw,circle,fill] (8) at ($(1)!.5!(3)$) {$8$};
				\node[draw,circle,fill] (9) at ($(5)!.5!(7)$) {$9$};
				
				\draw [line width=3pt,-] (1) -- (2);
				\draw [line width=3pt,-] (1) -- (7);
				\draw [line width=3pt,-] (1) -- (8);
				\draw [line width=3pt,-] (2) -- (3);
				\draw [line width=3pt,-] (3) -- (4);
				\draw [line width=3pt,-] (3) -- (8);
				\draw [line width=3pt,-] (4) -- (5);
				\draw [line width=3pt,-] (5) -- (6);
				\draw [line width=3pt,-] (5) -- (9);
				\draw [line width=3pt,-] (6) -- (7);
				\draw [line width=3pt,-] (7) -- (9);
				
				\end{tikzpicture}
			}
			\caption{Graph 13}
			\label[graph]{fig:graph13}
		\end{subfigure}
		\hfill
            \begin{subfigure}{0.175\textwidth}
			\centering
			\resizebox{\textwidth}{!}{
				\begin{tikzpicture}
				
				\node[draw,circle,fill] (1) at (0*360/7 +90: 5cm) {$1$};
				\node[draw,circle,fill] (2) at (1*360/7 +90: 5cm) {$2$};
				\node[draw,circle,fill] (3) at (2*360/7 +90: 5cm) {$3$};
				\node[draw,circle,fill] (4) at (3*360/7 +90: 5cm) {$4$};
				\node[draw,circle,fill] (5) at (4*360/7 +90: 5cm) {$5$};
				\node[draw,circle,fill] (6) at (5*360/7 +90: 5cm) {$6$};
                    \node[draw,circle,fill] (7) at (6*360/7 +90: 5cm) {$7$};
                    \node[draw,circle,fill] (8) at ($(3)!.5!(6)$) {$8$};
                    \node[draw,circle,fill] (9) at ($(2)!.5!(7) - (0,1.25)$) {$9$};

				\draw [line width=3pt,-] (1) -- (2);
                    \draw [line width=3pt,-] (1) -- (7);
				\draw [line width=3pt,-] (2) -- (3);
                    \draw [line width=3pt,-] (2) -- (9);
                    \draw [line width=3pt,-] (3) -- (4);
                    \draw [line width=3pt,-] (3) -- (8);
                    \draw [line width=3pt,-] (4) -- (5);
                    \draw [line width=3pt,-] (5) -- (6);
                    \draw [line width=3pt,-] (6) -- (7);
                    \draw [line width=3pt,-] (6) -- (8);
                    \draw [line width=3pt,-] (7) -- (9);

				\end{tikzpicture}
			}
			\caption{Graph 14}
			\label[graph]{fig:graph14}
		\end{subfigure}
            \hfill
		\begin{subfigure}{0.175\textwidth}
			\centering
			\resizebox{\textwidth}{!}{
				\begin{tikzpicture}
				
				\node[draw,circle,fill] (1) at (0*360/6 +120: 5cm) {$1$};
				\node[draw,circle,fill] (2) at (1*360/6 +120: 5cm) {$2$};
				\node[draw,circle,fill] (3) at (2*360/6 +120: 5cm) {$3$};
				\node[draw,circle,fill] (4) at (3*360/6 +120: 5cm) {$4$};
				\node[draw,circle,fill] (5) at (4*360/6 +120: 5cm) {$5$};
				\node[draw,circle,fill] (6) at (5*360/6 +120: 5cm) {$6$};
				\node[draw,circle,fill] (7) at ($(2)!.25!(5)$) {$7$};
				\node[draw,circle,fill] (8) at ($(2)!.5!(5)$) {$8$};
				\node[draw,circle,fill] (9) at ($(2)!.75!(5)$) {$9$};
				
				\draw [line width=3pt,-] (1) -- (2);
				\draw [line width=3pt,-] (1) -- (6);
				\draw [line width=3pt,-] (2) -- (3);
				\draw [line width=3pt,-] (2) -- (7);
				\draw [line width=3pt,-] (3) -- (4);
				\draw [line width=3pt,-] (4) -- (5);
				\draw [line width=3pt,-] (5) -- (6);
				\draw [line width=3pt,-] (5) -- (9);
				\draw [line width=3pt,-] (7) -- (8);
				\draw [line width=3pt,-] (8) -- (9);
				
				\end{tikzpicture}
			}
			\caption{Graph 15}
			\label[graph]{fig:graph15}
		\end{subfigure}
            \\
		\begin{subfigure}{0.175\textwidth}
			\centering
			\resizebox{\textwidth}{!}{
				\begin{tikzpicture}
				
				\node[draw,circle,fill] (1) at (0*360/6 +30: 5cm) {$1$};
				\node[draw,circle,fill] (2) at (1*360/6 +30: 5cm) {$2$};
				\node[draw,circle,fill] (3) at (2*360/6 +30: 5cm) {$3$};
				\node[draw,circle,fill] (4) at (3*360/6 +30: 5cm) {$4$};
				\node[draw,circle,fill] (5) at (4*360/6 +30: 5cm) {$5$};
				\node[draw,circle,fill] (6) at (5*360/6 +30: 5cm) {$6$};
				\node[draw,circle,fill] (7) at ($(4)!.5!(6)$) {$7$};
                    \node[draw,circle,fill] (8) at (0,0) {$8$};
                    \node[draw,circle,fill] (9) at ($(2)!.33!(8)$) {$9$};
                    \node[draw,circle,fill] (10) at ($(2)!.66!(8)$) {$1$};

				\draw [line width=3pt,-] (1) -- (2);
                    \draw [line width=3pt,-] (1) -- (6);
                    \draw [line width=3pt,-] (1) -- (8);
				\draw [line width=3pt,-] (2) -- (3);
                    \draw [line width=3pt,-] (2) -- (9);
				\draw [line width=3pt,-] (3) -- (4);
                    \draw [line width=3pt,-] (3) -- (8);
				\draw [line width=3pt,-] (4) -- (5);
                    \draw [line width=3pt,-] (4) -- (7);
				\draw [line width=3pt,-] (5) -- (6);
                    \draw [line width=3pt,-] (6) -- (7);
                    \draw [line width=3pt,-] (8) -- (10);
                    \draw [line width=3pt,-] (9) -- (10);

				\end{tikzpicture}
			}
			\caption{Graph 16}
			\label[graph]{fig:graph16}
		\end{subfigure}
		\caption{Currently known topologically minimal non-metrizable graphs}
		\label{fig:zoo}
	\end{figure}
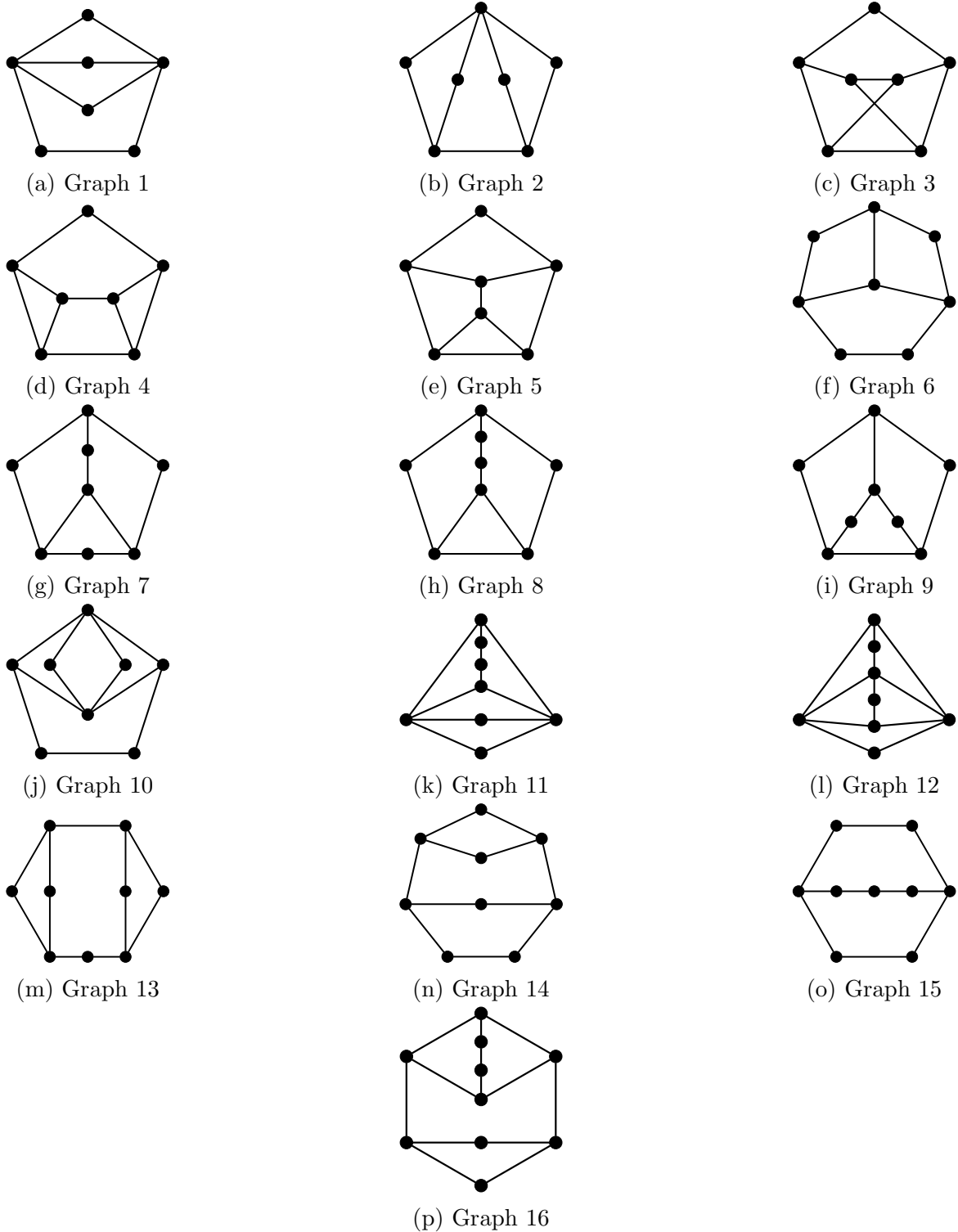
\newpage
\appendix \label{append:certificates}
\section{Certificates of non-metrizability}\label{append:certificates}
	For each graph $G$ in \Cref{fig:zoo} we give a path system in $G$ along with a system of inequalities a weight function inducing this path system must satisfy. In each case, these inequalities imply at least one edge in the graph must have a non-positive weight, showing the graph in not metrizable.\\
	\vspace{5mm}
	
	\noindent
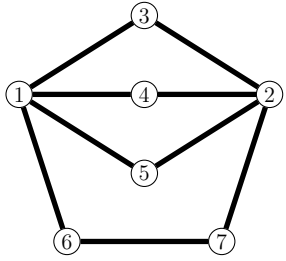
	\begin{minipage}[c]{0.4\textwidth}
		\centering
		\begin{tikzpicture}[scale=0.35, every node/.style={scale=0.35}]

		\node[draw,circle,minimum size=.5cm,inner sep=1pt] (2) at (1*360/5 +90: 5cm) [scale = 2]{$1$};
		\node[draw,circle,minimum size=.5cm,inner sep=1pt] (3) at (2*360/5 +90: 5cm)[scale = 2] {$6$};
		\node[draw,circle,minimum size=.5cm,inner sep=1pt] (4) at (3*360/5 +90: 5cm) [scale = 2]{$7$};
		\node[draw,circle,minimum size=.5cm,inner sep=1pt] (5) at (4*360/5 +90: 5cm) [scale = 2]{$2$};
		\node[draw,circle,minimum size=.5cm,inner sep=1pt] (1) at ($(2)!0.5!(5)$) [scale = 2]{$4$};
		\node[draw,circle,minimum size=.5cm,inner sep=1pt] (6) at ($(2)!0.5!(5) + (0,-3)$)[scale = 2] {$5$};
		\node[draw,circle,minimum size=.5cm,inner sep=1pt] (7) at  ($(2)!0.5!(5) + (0,3)$)[scale = 2] {$3$};

		\draw [line width=2pt,-] (1) -- (2);
		\draw [line width=2pt,-] (2) -- (3);
		\draw [line width=2pt,-] (3) -- (4);
		\draw [line width=2pt,-] (4) -- (5);
		\draw [line width=2pt,-] (5) -- (1);
		\draw [line width=2pt,-] (2) -- (6);
		\draw [line width=2pt,-] (5) -- (6);
		\draw [line width=2pt,-] (2) -- (7);
		\draw [line width=2pt,-] (5) -- (7);
		
		\end{tikzpicture}
	\end{minipage}
	\begin{minipage}[c]{0.4\textwidth}
		\centering
		\small
		\begin{equation*}
		\begin{gathered}
		(1,3,2), \ (1,6,7),  \ (2,7,6), \ (3,1,4) \ (3,2,5),\ (3,1,6),\\
		(3,1,6,7), \ (4,1,5),\ (4,2,7,6),\ (4,2,7),  \  (5,1,6), \ (5,2,7)
		\end{gathered}
		\end{equation*}
	\end{minipage}
	\\
	\vspace{4mm}
	\begin{minipage}{0.45\textwidth}
		\small
		\begin{equation*}
		\begin{split}
		w_{2,3} + w_{2,5} & \leq w_{1,3} + w_{1,5}\\
		w_{1,4} + w_{1,5} & \leq w_{2,4} + w_{2,5}\\
		w_{2,4} + w_{2,7} + w_{6,7} & \leq w_{1,4} + w_{1,6}\\
		w_{1,3} + w_{1,6} + w_{6,7} & \leq w_{2,3} + w_{2,7}
		\end{split}
		\end{equation*}
		
	\end{minipage}
	$\implies $
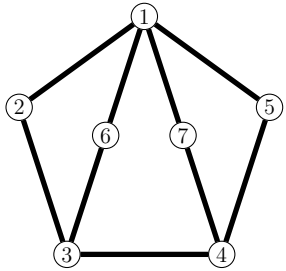
	\begin{minipage}{0.2\textwidth}
		\small
		$$w_{6,7} \leq 0$$
	\end{minipage}
	\begin{center}
		\line(1,0){400}
	\end{center}
	\begin{minipage}[c]{0.4\textwidth}
		\centering
		\begin{tikzpicture}[scale=0.35, every node/.style={scale=0.35}]

		\node[draw,circle,minimum size=.5cm,inner sep=1pt] (1) at (0*360/5 +90: 5cm) [scale=2]{$1$};
		\node[draw,circle,minimum size=.5cm,inner sep=1pt] (2) at (1*360/5 +90: 5cm) [scale=2]{$2$};
		\node[draw,circle,minimum size=.5cm,inner sep=1pt] (3) at (2*360/5 +90: 5cm) [scale=2]{$3$};
		\node[draw,circle,minimum size=.5cm,inner sep=1pt] (4) at (3*360/5 +90: 5cm) [scale=2]{$4$};
		\node[draw,circle,minimum size=.5cm,inner sep=1pt] (5) at (4*360/5 +90: 5cm) [scale=2]{$5$};
		
		\node[draw,circle,minimum size=.5cm,inner sep=1pt] (6) at ($(1)!0.5!(3)$)[scale=2] {$6$};
		\node[draw,circle,minimum size=.5cm,inner sep=1pt] (7) at  ($(1)!0.5!(4)$) [scale=2]{$7$};

		\draw [line width=2pt,-] (1) -- (2);
		\draw [line width=2pt,-] (2) -- (3);
		\draw [line width=2pt,-] (3) -- (4);
		\draw [line width=2pt,-] (4) -- (5);
		\draw [line width=2pt,-] (5) -- (1);
		\draw [line width=2pt,-] (1) -- (6);
		\draw [line width=2pt,-] (3) -- (6);
		\draw [line width=2pt,-] (1) -- (7);
		\draw [line width=2pt,-] (4) -- (7);
		
		\end{tikzpicture}
	\end{minipage}
	\begin{minipage}[c]{0.4\textwidth}
		\small
		\begin{equation*}
		\begin{gathered}
		 (1,6,3), \ (1,7,4),  \ (2,3,4), \ (2,3,4,5), \\ 
            (2,1,6), \ (2,1,7), \ (3,4,5), \ (3,4,7), \\
            (4,3,6), \ (5,4,3,6), \ (5,1,7), \ (6,3,4,7)
		\end{gathered}
		\end{equation*}
	\end{minipage}
	\\
	\vspace{4mm}
	\begin{minipage}{0.45\textwidth}
		\small
		\begin{equation*}
		\begin{split}
		w_{2,3} + w_{3,4} +w_{4,5}& \leq w_{1,2} + w_{1,5}\\
		w_{1,2} + w_{1,6} & \leq w_{2,3} + w_{3,6}\\
		w_{1,5} + w_{1,7} & \leq w_{4,5} + w_{4,7}\\
		w_{3,6} + w_{3,4} + w_{4,7} & \leq w_{1,6} + w_{1,7}
		\end{split}
		\end{equation*}
	\end{minipage}
	$\implies $
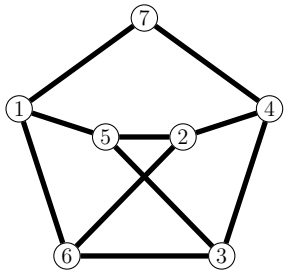
	\begin{minipage}{0.2\textwidth}
		\small $$w_{3,4} \leq 0$$
	\end{minipage}
	\begin{center}
		\line(1,0){400}
	\end{center}
	\begin{minipage}[c]{0.4\textwidth}
		\centering
		\begin{tikzpicture}[scale=0.35, every node/.style={scale=0.35}]

		\node[draw,circle,minimum size=.5cm,inner sep=1pt] (1) at (0*360/5 +90: 5cm) [scale=2]{$7$};
		\node[draw,circle,minimum size=.5cm,inner sep=1pt] (2) at (1*360/5 +90: 5cm) [scale=2]{$1$};
		\node[draw,circle,minimum size=.5cm,inner sep=1pt] (3) at (2*360/5 +90: 5cm)[scale=2] {$6$};
		\node[draw,circle,minimum size=.5cm,inner sep=1pt] (4) at (3*360/5 +90: 5cm) [scale=2]{$3$};
		\node[draw,circle,minimum size=.5cm,inner sep=1pt] (5) at (4*360/5 +90: 5cm) [scale=2]{$4$};
		
		\node[draw,circle,minimum size=.5cm,inner sep=1pt] (6) at ($(1)!0.5!(3)$) [scale=2]{$5$};
		\node[draw,circle,minimum size=.5cm,inner sep=1pt] (7) at  ($(1)!0.5!(4)$) [scale=2]{$2$};

		\draw [line width=2pt,-] (1) -- (2);
		\draw [line width=2pt,-] (2) -- (3);
		\draw [line width=2pt,-] (3) -- (4);
		\draw [line width=2pt,-] (4) -- (5);
		\draw [line width=2pt,-] (5) -- (1);
		\draw [line width=2pt,-] (6) -- (7);
		\draw [line width=2pt,-] (2) -- (6);
		\draw [line width=2pt,-] (5) -- (7);
		\draw [line width=2pt,-] (3) -- (7);
		\draw [line width=2pt,-] (4) -- (6);
		
		\end{tikzpicture}
	\end{minipage}
	\begin{minipage}[c]{0.4\textwidth}
		\centering
		\small
		\begin{equation*}
		\begin{gathered}
		(1,7,4,2), \ (1,6,3), \ (1,7,4), \\ 
		(2,6,3),   (2,4,7),  \ (3,4,7),\ (4,3,5), \\ 
                (4,3,6),\ (5,1,6), \ (5,3,4,7), \ (6,1,7)
		\end{gathered}
		\end{equation*}
	\end{minipage}
	\\
	\vspace{5mm}
	\begin{minipage}{0.45\textwidth}
		\small
		\begin{equation*}
		\begin{split}
		w_{2,6} + w_{3,6} & \leq w_{2,4} + w_{3,4}\\
		w_{1,5} + w_{1,6} & \leq w_{3,5} + w_{3,6}\\
		w_{1,7} + w_{4,7} + w_{2,4} & \leq w_{1,6} + w_{2,6}\\
		w_{3,5} + w_{3,4} + w_{4,7} & \leq w_{1,5} + w_{1,7}
		\end{split}
		\end{equation*}
	\end{minipage}
	$\implies $
	\begin{minipage}{0.2\textwidth}
		\small
		$$w_{4,7} \leq 0$$
	\end{minipage}
	\begin{center}
		\line(1,0){400}
	\end{center}
	\begin{minipage}[c]{0.4\textwidth}
		\centering
		\begin{tikzpicture}[scale=0.35, every node/.style={scale=0.35}]

		\node[draw,circle,minimum size=.5cm,inner sep=1pt] (1) at (0*360/5 +90: 5cm)[scale =2] {$1$};
		\node[draw,circle,minimum size=.5cm,inner sep=1pt] (2) at (1*360/5 +90: 5cm)[scale =2] {$2$};
		\node[draw,circle,minimum size=.5cm,inner sep=1pt] (3) at (2*360/5 +90: 5cm)[scale =2] {$3$};
		\node[draw,circle,minimum size=.5cm,inner sep=1pt] (4) at (3*360/5 +90: 5cm)[scale =2] {$4$};
		\node[draw,circle,minimum size=.5cm,inner sep=1pt] (5) at (4*360/5 +90: 5cm)[scale =2] {$5$};
		
		\node (m1) at ($(2)!.5! (3)$) {};
		\node (m2) at ($(4)!.5! (5)$) {};
		\node[draw,circle,minimum size=.5cm,inner sep=1pt] (6) at ($(m1)!.8!270:(2)$)[scale =2] {$6$};
		\node[draw,circle,minimum size=.5cm,inner sep=1pt] (7) at  ($(m2)!.8!90:(5)$)[scale =2] {$7$};

		\draw [line width=2pt,-] (1) -- (2);
		\draw [line width=2pt,-] (2) -- (3);
		\draw [line width=2pt,-] (3) -- (4);
		\draw [line width=2pt,-] (4) -- (5);
		\draw [line width=2pt,-] (5) -- (1);
		\draw [line width=2pt,-] (6) -- (7);
		\draw [line width=2pt,-] (2) -- (6);
		\draw [line width=2pt,-] (5) -- (7);
		\draw [line width=2pt,-] (3) -- (6);
		\draw [line width=2pt,-] (4) -- (7);

		\end{tikzpicture}
	\end{minipage}
	\begin{minipage}[c]{0.4\textwidth}
		\centering
		\small
		\begin{equation*}
		\begin{gathered}
		 (1,2,3), \ (1,2,3,4), \ (1,2,6), \ (1,5,7), \\
		(2,3,4), \ (2,1,5), \ (2,1,5,7),  \ (3,4,5),  \\
		(3,6,7),  \ (4,7,6), \ (5,1,2,6)
		\end{gathered}
		\end{equation*}
	\end{minipage}
	\\
	\vspace{5mm}
	\begin{minipage}{0.45\textwidth}
		\small
		\begin{equation*}
		\begin{split}
		w_{1,2} + w_{2,3} + w_{3,4} & \leq w_{1,5} +  w_{4,5}\\
		w_{1,2} + w_{1,5} + w_{5,7} & \leq w_{2,6} + w_{6,7}\\
		w_{1,5} + w_{1,2} + w_{2,6} & \leq w_{5,7} + w_{6,7}\\
		w_{3,6} + w_{6,7} & \leq w_{3,4} + w_{4,7}\\
		w_{4,7} + w_{6,7} & \leq w_{3,4} + w_{3,6}\\
		w_{3,4} + w_{4,5}  & \leq w_{1,5} + w_{1,2} + w_{2,3}
		\end{split}
		\end{equation*}
	\end{minipage}
	$\implies $
	\begin{minipage}{0.2\textwidth}
		\small
		$$w_{1,2} \leq 0$$
	\end{minipage}
	\begin{center}
		\line(1,0){400}
	\end{center}
	\begin{minipage}[c]{0.4\textwidth}
		\centering
		\begin{tikzpicture}[scale=0.35, every node/.style={scale=0.35}]

		\node[draw,circle,minimum size=.5cm,inner sep=1pt] (1) at (0*360/5 +90: 5cm)[scale=2] {$1$};
		\node[draw,circle,minimum size=.5cm,inner sep=1pt] (2) at (1*360/5 +90: 5cm)[scale=2] {$2$};
		\node[draw,circle,minimum size=.5cm,inner sep=1pt] (3) at (2*360/5 +90: 5cm)[scale=2] {$3$};
		\node[draw,circle,minimum size=.5cm,inner sep=1pt] (4) at (3*360/5 +90: 5cm)[scale=2] {$4$};
		\node[draw,circle,minimum size=.5cm,inner sep=1pt] (5) at (4*360/5 +90: 5cm)[scale=2] {$5$};
		
		\node[draw,circle,minimum size=.5cm,inner sep=1pt] (6) at ($(2)!.5!(5)+(0,-1)$)[scale=2] {$6$};
		\node[draw,circle,minimum size=.5cm,inner sep=1pt] (7) at ($(2)!.5!(5) + (0,-3)$)[scale=2] {$7$};

		\draw [line width=2pt,-] (1) -- (2);
		\draw [line width=2pt,-] (2) -- (3);
		\draw [line width=2pt,-] (3) -- (4);
		\draw [line width=2pt,-] (4) -- (5);
		\draw [line width=2pt,-] (5) -- (1);
		\draw [line width=2pt,-] (6) -- (7);
		\draw [line width=2pt,-] (2) -- (6);
		\draw [line width=2pt,-] (5) -- (6);
		\draw [line width=2pt,-] (3) -- (7);
		\draw [line width=2pt,-] (4) -- (7);

		\end{tikzpicture}
	\end{minipage}
	\begin{minipage}[c]{0.4\textwidth}
		\centering
		\small
		\begin{equation*}
		\begin{gathered}
		(1,5,4,3), \ (1,5,4),  \ (1,5,6), \ (1,5,6,7), \\
		(2,1,5,4), \ (2,1,5), \ (2,3,7), \ (3,4,5)\\
		(3,2,6),\ (4,7,6), \ (5,6,7)
		\end{gathered}
		\end{equation*}
	\end{minipage}
	\\
	\vspace{5mm}
	\begin{minipage}{0.45\textwidth}
		\small
		\begin{equation*}
		\begin{split}
		w_{1,5} + w_{4,5} + w_{3,4} & \leq w_{1,2} + w_{2,3}\\
		w_{1,2} + w_{1,5} + w_{4,5} & \leq w_{2,3} + w_{3,4}\\
		w_{2,3} + w_{3,7}  & \leq w_{2,6} + w_{6,7}\\
		w_{2,3} + w_{2,6}  & \leq w_{3,7} + w_{6,7}\\
		w_{5,6} + w_{6,7} & \leq w_{4,5} + w_{4,7}\\
		w_{4,7} + w_{6,7}  & \leq w_{4,5} + w_{5,6}
		\end{split}
		\end{equation*}
	\end{minipage}
	$\implies $
	\begin{minipage}{0.2\textwidth}
		\small
		$$w_{1,5} \leq 0$$
	\end{minipage}
	\begin{center}
		\line(1,0){400}
	\end{center}
	\begin{minipage}[c]{0.4\textwidth}
		\centering
		\begin{tikzpicture}[scale=0.35, every node/.style={scale=0.35}]

		\node[draw,circle,minimum size=.5cm,inner sep=1pt] (1) at (0*360/5 +90: 5cm) [scale =2]{$1$};
		\node[draw,circle,minimum size=.5cm,inner sep=1pt] (2) at (-1*360/5 +90: 5cm) [scale =2]{$2$};
		\node[draw,circle,minimum size=.5cm,inner sep=1pt] (3) at (-2*360/5 +90: 5cm) [scale =2]{$3$};
		\node[draw,circle,minimum size=.5cm,inner sep=1pt] (4) at (-3*360/5 +90: 5cm)[scale =2] {$4$};
		\node[draw,circle,minimum size=.5cm,inner sep=1pt] (5) at (-4*360/5 +90: 5cm) [scale =2]{$5$};
		\node[draw,circle,minimum size=.5cm,inner sep=1pt] (6) at (0,0) [scale =2]{$6$};
		\node[draw,circle,minimum size=.5cm,inner sep=1pt] (7) at ($(4)!.5!(6)$) [scale =2]{$7$};
		\node[draw,circle,minimum size=.5cm,inner sep=1pt] (8) at ($(3)!.5!(6)$) [scale =2]{$8$};

		\draw [line width=2pt,-] (1) -- (2);
		\draw [line width=2pt,-] (1) -- (5);
		\draw [line width=2pt,-] (1) -- (6);
		\draw [line width=2pt,-] (2) -- (3);
		\draw [line width=2pt,-] (3) -- (4);
		\draw [line width=2pt,-] (3) -- (8);
		\draw [line width=2pt,-] (4) -- (5);
		\draw [line width=2pt,-] (4) -- (7);
		\draw [line width=2pt,-] (6) -- (7);
		\draw [line width=2pt,-] (6) -- (8);

		\end{tikzpicture}
	\end{minipage}
	\begin{minipage}[c]{0.4\textwidth}
		\centering
		\small
		\begin{equation*}
		\begin{gathered}
		(1,5,4,3), \ (1,5,4),  \ (1,6,7), \ (1,6,8), \\
		(2,3,4), \ (2,3,4,5), \ (2,1,6), \ (2,1,6,7), \ (2,1,6,8), \ (3,4,5), \\ 
		(3,8,6), \ (3,4,7), \ (4,5,1,6),  \ (4,3,8), \\
		(5,1,6), \ (5,4,7), \ (5,1,6,8),  \ (7,4,3,8)
		\end{gathered}
		\end{equation*}
	\end{minipage}
	\\
	\vspace{5mm}
	\begin{minipage}{0.45\textwidth}
		\small
		\begin{equation*}
		\begin{split}
		w_{4,7} + w_{3,4} + w_{3,8} & \leq  w_{6,7} +w_{6,8}\\
		w_{2,3} + w_{3,4} + w_{4,5} & \leq  w_{1,2} +w_{1,5}\\
		w_{1,5} + w_{1,6} + w_{6,8} & \leq  w_{4,5} +w_{3,4}+w_{3,8}\\
		w_{1,2} + w_{1,6} + w_{6,7} & \leq  w_{2,3} +w_{3,4}+w_{4,7}
		\end{split}
		\end{equation*}
	\end{minipage}
	$\implies $
	\begin{minipage}{0.2\textwidth}
		\small
		$$w_{1,6} \leq 0$$
	\end{minipage}
	\begin{center}
		\line(1,0){400}
	\end{center}
	\begin{minipage}[c]{0.4\textwidth}
		\centering
		\begin{tikzpicture}[scale=0.34, every node/.style={scale=0.34}]

		\node[draw,circle,minimum size=.5cm,inner sep=1pt] (1) at (0*360/7 +90: 5cm)[scale =2] {$1$};
		\node[draw,circle,minimum size=.5cm,inner sep=1pt] (2) at (1*360/7 +90: 5cm)[scale =2] {$2$};
		\node[draw,circle,minimum size=.5cm,inner sep=1pt] (3) at (2*360/7 +90: 5cm) [scale =2]{$3$};
		\node[draw,circle,minimum size=.5cm,inner sep=1pt] (4) at (3*360/7 +90: 5cm) [scale =2]{$4$};
		\node[draw,circle,minimum size=.5cm,inner sep=1pt] (5) at (4*360/7 +90: 5cm) [scale =2]{$5$};
		\node[draw,circle,minimum size=.5cm,inner sep=1pt] (6) at (5*360/7 +90: 5cm) [scale =2]{$6$};
		\node[draw,circle,minimum size=.5cm,inner sep=1pt] (7) at (6*360/7 +90: 5cm) [scale =2]{$7$};
		
		\node[draw,circle,minimum size=.5cm,inner sep=1pt] (8) at (0,0) [scale =2]{$8$};

		\draw [line width=2pt,-] (1) -- (2);
		\draw [line width=2pt,-] (2) -- (3);
		\draw [line width=2pt,-] (3) -- (4);
		\draw [line width=2pt,-] (4) -- (5);
		\draw [line width=2pt,-] (5) -- (6);
		\draw [line width=2pt,-] (6) -- (7);
		\draw [line width=2pt,-] (7) -- (1);
		\draw [line width=2pt,-] (1) -- (8);
		\draw [line width=2pt,-] (3) -- (8);
		\draw [line width=2pt,-] (6) -- (8);

		\end{tikzpicture}
	\end{minipage}
	\begin{minipage}[c]{0.4\textwidth}
		\centering
		\small
		\begin{equation*}
		\begin{gathered}
		(1,2,3), \ (1,7,6,5,4), \ (1,7,6,5), \ (1,7,6),\\ 
		(2,3,4), \ (2,3,4,5), \ (2,3,4,5,6), \ (2,1,7), \ (2,1,8),  \\
		(3,4,5), \ (3,4,5,6), \ (3,2,1,7),  \ (4,5,6), \ (4,5,6,7), \\
		(4,3,8), \  (5,6,7), \ (5,4,3,8),  \ (7,6,8)
		\end{gathered}
		\end{equation*}
	\end{minipage}
	\\
	\vspace{5mm}
	\begin{minipage}{0.45\textwidth}
		\small
		\begin{equation*}
		\begin{split}
		w_{4,5} + w_{3,4}+w_{3,8} & \leq w_{5,6} + w_{6,8}\\
		w_{1,2} + w_{1,8} & \leq w_{2,3} + w_{3,8}\\
		w_{6,7} + w_{6,8} & \leq w_{1,7} + w_{1,8}\\
		w_{1,7}+w_{6,7} + w_{5,6} + w_{4,5} & \leq w_{1,2} + w_{2,3} + w_{3,4}\\
		w_{2,3}+w_{3,4} + w_{4,5} + w_{5,6} & \leq w_{1,2} + w_{1,7} + w_{6,7}\\
		w_{2,3} + w_{1,2} + w_{1,7} & \leq w_{3,4} + w_{4,5} + w_{5,6} + w_{6,7}
		\end{split}
		\end{equation*}
	\end{minipage}
	$\hspace{10mm}\implies $
	\begin{minipage}{0.2\textwidth}
		\small
		$$w_{4,5} \leq 0$$
	\end{minipage}
	\begin{center}
		\line(1,0){400}
	\end{center}
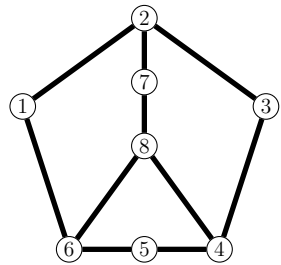
	\begin{minipage}[c]{0.4\textwidth}
		\centering
		\begin{tikzpicture}[scale=0.34, every node/.style={scale=0.34}]
		
		\node[draw,circle,minimum size=.5cm,inner sep=1pt] (2) at (0*360/5 +90: 5cm) [scale =2]{$2$};
		\node[draw,circle,minimum size=.5cm,inner sep=1pt] (1) at (1*360/5 +90: 5cm) [scale =2]{$1$};
		\node[draw,circle,minimum size=.5cm,inner sep=1pt] (6) at (2*360/5 +90: 5cm) [scale =2]{$6$};
		\node[draw,circle,minimum size=.5cm,inner sep=1pt] (4) at (3*360/5 +90: 5cm)[scale =2] {$4$};
		\node[draw,circle,minimum size=.5cm,inner sep=1pt] (3) at (4*360/5 +90: 5cm)[scale =2] {$3$};
		\node[draw,circle,minimum size=.5cm,inner sep=1pt] (8) at (0,0)[scale =2] {$8$};
		\node[draw,circle,minimum size=.5cm,inner sep=1pt] (5) at ($(4)!.5!(6)$)[scale =2] {$5$};
		\node[draw,circle,minimum size=.5cm,inner sep=1pt] (7) at ($(2)!.5!(8)$)[scale =2] {$7$};

		\draw [line width=2pt,-] (1) -- (2);
		\draw [line width=2pt,-] (1) -- (6);
		\draw [line width=2pt,-] (2) -- (3);
		\draw [line width=2pt,-] (2) -- (7);
		\draw [line width=2pt,-] (3) -- (4);
		\draw [line width=2pt,-] (4) -- (5);
		\draw [line width=2pt,-] (4) -- (8);
		\draw [line width=2pt,-] (5) -- (6);
		\draw [line width=2pt,-] (6) -- (8);
		\draw [line width=2pt,-] (7) -- (8);

		\end{tikzpicture}
	\end{minipage}
	\begin{minipage}[c]{0.4\textwidth}
		\centering
		\small
		\begin{equation*}
		\begin{gathered}
		 (1,2,3), \ (1,2,3,4), \ (1,6,5), \ (1,6,8,7), \ (1,6,8),\\
		 (2,3,4), \ (2,1,6,5), \ (2,1,6),\ (2,7,8),  \\ 
		(3,4,5), \ (3,4,5,6), \ (3,2,7), \ (3,2,7,8),  \ (4,5,6), \ (4,8,7), \\ 
		(5,4,8,7), \ (5,4,8)
		\end{gathered}
		\end{equation*}
	\end{minipage}
	\\
	\vspace{5mm}
	\begin{minipage}{0.45\textwidth}
		\small
		\begin{equation*}
		\begin{split}
		w_{1,6} + w_{6,8} + w_{7,8}  & \leq w_{1,2} + w_{2,7}\\
		w_{2,3} + w_{2,7} + w_{7,8}  & \leq w_{3,4} + w_{4,8}\\
		w_{4,5} + w_{4,8}  & \leq w_{5,6} + w_{6,8}\\
		w_{1,2} + w_{2,3} + w_{3,4}  & \leq w_{1,6} + w_{5,6} + w_{4,5}\\
		w_{1,2} + w_{1,6} + w_{5,6}  & \leq w_{2,3} + w_{3,4} + w_{4,5}\\
		w_{3,4} + w_{4,5} + w_{5,6}  & \leq w_{2,3} + w_{1,2} + w_{1,6}
		\end{split}
		\end{equation*}
	\end{minipage}
	$\implies $
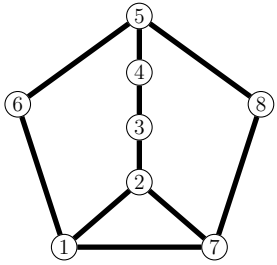
	\begin{minipage}{0.2\textwidth}
		\small
		$$w_{7,8} \leq 0$$
	\end{minipage}
	\begin{center}
		\line(1,0){400}
	\end{center}
	\begin{minipage}[c]{0.4\textwidth}
		\centering
		\begin{tikzpicture}[scale=0.34, every node/.style={scale=0.34}]

		\node[draw,circle,minimum size=.5cm,inner sep=1pt] (5) at (0*360/5 +90: 5cm) [scale =2] {$5$};
		\node[draw,circle,minimum size=.5cm,inner sep=1pt] (6) at (1*360/5 +90: 5cm) [scale =2] {$6$};
		\node[draw,circle,minimum size=.5cm,inner sep=1pt] (1) at (2*360/5 +90: 5cm) [scale =2] {$1$};
		\node[draw,circle,minimum size=.5cm,inner sep=1pt] (7) at (3*360/5 +90: 5cm)  [scale =2]{$7$};
		\node[draw,circle,minimum size=.5cm,inner sep=1pt] (8) at (4*360/5 +90: 5cm)  [scale =2]{$8$};
		\node[draw,circle,minimum size=.5cm,inner sep=1pt] (2) at (0,-1.5)  [scale =2]{$2$};
		\node[draw,circle,minimum size=.5cm,inner sep=1pt] (3) at ($(2)!.33!(5)$)  [scale =2]{$3$};
		\node[draw,circle,minimum size=.5cm,inner sep=1pt] (4) at ($(2)!.66!(5)$) [scale =2] {$4$};

		\draw [line width=2pt,-] (1) -- (2);
		\draw [line width=2pt,-] (1) -- (6);
		\draw [line width=2pt,-] (1) -- (7);
		\draw [line width=2pt,-] (2) -- (3);
		\draw [line width=2pt,-] (2) -- (7);
		\draw [line width=2pt,-] (3) -- (4);
		\draw [line width=2pt,-] (4) -- (5);
		\draw [line width=2pt,-] (5) -- (6);
		\draw [line width=2pt,-] (5) -- (8);
		\draw [line width=2pt,-] (7) -- (8);

		\end{tikzpicture}
	\end{minipage}
	\begin{minipage}[c]{0.4\textwidth}
		\centering
		\small
		\begin{equation*}
		\begin{gathered}
		 (1,6,5,4,3), \ (1,6,5,4), \ (1,6,5), \ (1,7,8), \\ 
		(2,3,4), \ (2,3,4,5), \ (2,1,6), \ (2,3,4,5,8), \\
		(3,4,5), \ (3,4,5,6), \ (3,2,7), \ (3,4,5,8), \ (4,5,6), \ (4,3,2,7), \\
		(4,5,8), \ (5,8,7),  \ (6,5,8,7), \ (6,5,8)
		\end{gathered}
		\end{equation*}
	\end{minipage}
	\\
	\vspace{4mm}
	\begin{minipage}{0.45\textwidth}
		\small
		\begin{equation*}
		\begin{split}
		w_{1,6} + w_{5,6} + w_{4,5} + w_{3,4}  & \leq w_{1,2} + w_{2,3}\\
		w_{2,3} + w_{3,4} + w_{4,5} + w_{5,8} & \leq w_{2,7} +w_{7,8}\\
		w_{5,6} + w_{5,8} + w_{7,8} & \leq w_{1,6} +w_{1,7}\\
		w_{3,4}+ w_{2,3} + w_{2,7}  & \leq w_{4,5} +w_{5,8}+ w_{7,8} \\
		w_{1,7} + w_{7,8} &\leq  w_{1,6} + w_{5,6} + w_{5,8}\\
		w_{1,2} + w_{1,6} &\leq  w_{2,3} + w_{3,4}+w_{4,5} +w_{5,6}
		\end{split}
		\end{equation*}
	\end{minipage}
	$\hspace{10mm}\implies $
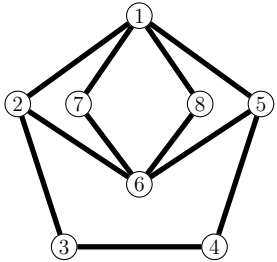
	\begin{minipage}{0.2\textwidth}
		\small
		$$w_{3,4} \leq 0$$
	\end{minipage}
	\begin{center}
		\line(1,0){400}
	\end{center}
        \begin{minipage}[c]{0.4\textwidth}
		\centering
		\begin{tikzpicture}[scale=0.34, every node/.style={scale=0.34}]

		\node[draw,circle,minimum size=.5cm,inner sep=1pt] (1) at (0*360/5 +90: 5cm) [scale =2] {$1$};
		\node[draw,circle,minimum size=.5cm,inner sep=1pt] (2) at (1*360/5 +90: 5cm) [scale =2] {$2$};
		\node[draw,circle,minimum size=.5cm,inner sep=1pt] (3) at (2*360/5 +90: 5cm) [scale =2] {$3$};
		\node[draw,circle,minimum size=.5cm,inner sep=1pt] (4) at (3*360/5 +90: 5cm)  [scale =2]{$4$};
		\node[draw,circle,minimum size=.5cm,inner sep=1pt] (5) at (4*360/5 +90: 5cm)  [scale =2]{$5$};
		\node[draw,circle,minimum size=.5cm,inner sep=1pt] (6) at (0,-1.6) [scale =2]{$6$};
		\node[draw,circle,minimum size=.5cm,inner sep=1pt] (7) at ($(2)!.25!(5)$) [scale =2]{$7$};
		\node[draw,circle,minimum size=.5cm,inner sep=1pt] (8) at ($(2)!.75!(5)$) [scale =2]{$8$};

				\draw [line width=2pt,-] (1) -- (2);
                    \draw [line width=2pt,-] (1) -- (5);
                    \draw [line width=2pt,-] (1) -- (7);
                    \draw [line width=2pt,-] (1) -- (8);
                    \draw [line width=2pt,-] (2) -- (3);
                    \draw [line width=2pt,-] (2) -- (6);
                    \draw [line width=2pt,-] (3) -- (4);
                    \draw [line width=2pt,-] (4) -- (5);
                    \draw [line width=2pt,-] (5) -- (6);
                    \draw [line width=2pt,-] (6) -- (7);
                    \draw [line width=2pt,-] (6) -- (8);

		\end{tikzpicture}
	\end{minipage}
	\begin{minipage}[c]{0.4\textwidth}
		\centering
		\small
		\begin{equation*}
		\begin{gathered}
		 (1, 5, 4, 3), \ (1, 5, 4), \ (1, 5, 6), \ (2, 3, 4), \ (2, 6, 5),  (2, 1, 7), \\ (2, 1, 8), \ (3, 4, 5), \ (3, 2, 6), \ (3, 4, 5, 1, 7), \ (3, 4, 5, 1, 8), \\ (4, 3, 2, 6), \ (4, 5, 1, 7), \ (4, 5, 1, 8), \ (5, 1, 7), \ (5, 1, 8),  \ (7, 6, 8)
		\end{gathered}
		\end{equation*}
	\end{minipage}
	\\
	\vspace{4mm}
	\begin{minipage}{0.45\textwidth}
		\small
		\begin{equation*}
		\begin{split}
		w_{1,5} + w_{1,8} + w_{3,4} + w_{4,5}  & \leq w_{2,3} + 
            w_{2,6}+w_{6,8}\\
		w_{3,4}+w_{2,3} + w_{2,6}  & \leq w_{4,5} +w_{5,6}\\
		w_{2,6} + w_{5,6} & \leq w_{1,2} +w_{1,5}\\
		w_{1,2}+ w_{1,7}  & \leq w_{2,6} +w_{6,7} \\
		w_{6,7} + w_{6,8} &\leq  w_{1,7} + w_{1,8}\\
		\end{split}
		\end{equation*}
	\end{minipage}
	$\implies $
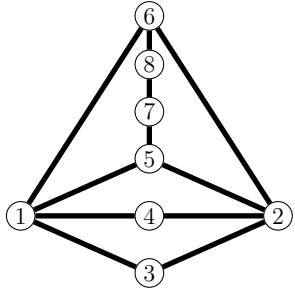
	\begin{minipage}{0.2\textwidth}
		\small
		$$w_{3,4} \leq 0$$
	\end{minipage}
	\begin{center}
		\line(1,0){400}
	\end{center}
        \begin{minipage}[c]{0.4\textwidth}
		\centering
		\begin{tikzpicture}[scale=0.38, every node/.style={scale=0.38}]

            \node[draw,circle,minimum size=.5cm,inner sep=1pt]  (1) at (-4.5,0) [scale=2]{$1$};
				\node[draw,circle,minimum size=.5cm,inner sep=1pt] (2) at (4.5,0)  [scale =2]  {$2$};
				\node[draw,circle,minimum size=.5cm,inner sep=1pt] (4) at ($(1)!.5!(2)$)  [scale =2] {$4$};
                    \node[draw,circle,minimum size=.5cm,inner sep=1pt] (3) at ($(1)!.5!(2)+ (0,-2)$)  [scale =2] {$3$};
                    \node[draw,circle,minimum size=.5cm,inner sep=1pt] (5) at ($(1)!.5!(2)+ (0,2)$)  [scale =2] {$5$};
                    \node[draw,circle,minimum size=.5cm,inner sep=1pt] (6) at ($(1)!.5!(2)+ (0,7)$)  [scale =2] {$6$};
                    \node[draw,circle,minimum size=.5cm,inner sep=1pt] (7) at ($(5)!.33!(6)$)  [scale =2] {$7$};
                    \node[draw,circle,minimum size=.5cm,inner sep=1pt] (8) at ($(5)!.66!(6)$)  [scale =2] {$8$};

				\draw [line width=2pt,-] (1) -- (3);
                    \draw [line width=2pt,-] (1) -- (4);
                    \draw [line width=2pt,-] (1) -- (5);
                    \draw [line width=2pt,-] (1) -- (6);
                    \draw [line width=2pt,-] (2) -- (3);
                    \draw [line width=2pt,-] (2) -- (4);
                    \draw [line width=2pt,-] (2) -- (5);
                    \draw [line width=2pt,-] (2) -- (6);
                    \draw [line width=2pt,-] (5) -- (7);
                    \draw [line width=2pt,-] (7) -- (8);
                    \draw [line width=2pt,-] (6) -- (8);

		\end{tikzpicture}
	\end{minipage}
	\begin{minipage}[c]{0.4\textwidth}
		\centering
		\small
		\begin{equation*}
		\begin{gathered}
		(1, 6, 2), \ (1, 6, 8, 7), \ (1, 6, 8), \ (2, 5, 7), \ (2, 5, 7, 8), \\
  (3, 1, 4), \ (3, 2, 5), \ (3, 2, 6), \ (3, 2, 5, 7), \ (3, 2, 5, 7, 8), \\
  (4, 2, 5), \ (4, 2, 6), \ (4, 2, 5, 7), \
  (4, 2, 5, 7, 8), \\ (5, 1, 6), \ (5, 7, 8), \ (6, 8, 7)
		\end{gathered}
		\end{equation*}
	\end{minipage}
	\\
	\vspace{5mm}
	\begin{minipage}{0.45\textwidth}
		\small
		\begin{equation*}
		\begin{split}
		w_{1,6} + w_{6,8} + w_{7,8}& \leq w_{1,5} + w_{5,7}\\
            w_{2,3} + w_{2,5} + w_{5,7} + w_{7,8}& \leq w_{1,3} + w_{1,6} + w_{6,8}\\
            w_{1,3} + w_{1,4}& \leq w_{2,3} + w_{2,4}\\
            w_{2,4} + w_{2,6}& \leq w_{1,4} + w_{1,6}\\
            w_{1,5} + w_{1,6}& \leq w_{2,5} + w_{2,6}
		\end{split}
		\end{equation*}
	\end{minipage}
	\hspace{10mm}$\implies $
	\begin{minipage}{0.2\textwidth}
		\small
		$$w_{7,8} \leq 0$$
	\end{minipage}
	\begin{center}
		\line(1,0){400}
	\end{center}
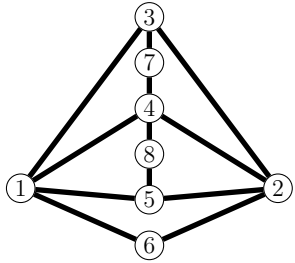
        \begin{minipage}[c]{0.4\textwidth}
		\centering
		\begin{tikzpicture}[scale=0.38, every node/.style={scale=0.38}]
            
            \node[draw,circle,minimum size=.5cm,inner sep=1pt] (1) at (-4.5,0)[scale =2] {$1$};
            \node[draw,circle,minimum size=.5cm,inner sep=1pt] (2) at (4.5,0)[scale =2] {$2$};
            \node[draw,circle,minimum size=.5cm,inner sep=1pt] (3) at ($(1)!.5!(2)+ (0,6)$)[scale =2] {$3$};
            \node[draw,circle,minimum size=.5cm,inner sep=1pt](6) at ($(1)!.5!(2)+ (0,-2)$)[scale =2] {$6$};
            \node[draw,circle,minimum size=.5cm,inner sep=1pt] (4) at ($(3)!.4!(6)$)[scale =2] {$4$};
            \node[draw,circle,minimum size=.5cm,inner sep=1pt](5) at ($(3)!.8!(6)$)[scale =2] {$5$};
            \node[draw,circle,minimum size=.5cm,inner sep=1pt] (7) at ($(3)!.2!(6)$)[scale =2] {$7$};
            \node[draw,circle,minimum size=.5cm,inner sep=1pt] (8) at ($(3)!.6!(6)$)[scale =2] {$8$};

            \draw [line width=2pt,-] (1) -- (3);
            \draw [line width=2pt,-] (1) -- (4);
            \draw [line width=2pt,-] (1) -- (5);
            \draw [line width=2pt,-] (1) -- (6);
            \draw [line width=2pt,-] (2) -- (3);
            \draw [line width=2pt,-] (2) -- (4);
            \draw [line width=2pt,-] (2) -- (5);
            \draw [line width=2pt,-] (2) -- (6);
            \draw [line width=2pt,-] (3) -- (7);
            \draw [line width=2pt,-] (4) -- (7);
            \draw [line width=2pt,-] (4) -- (8);
            \draw [line width=2pt,-] (5) -- (8);

		\end{tikzpicture}
	\end{minipage}
	\begin{minipage}[c]{0.4\textwidth}
		\centering
		\small
		\begin{equation*}
		\begin{gathered}
		(1, 5, 2), \ (1, 3, 7), \ (1, 5, 8), \ (2, 4, 7), \ (2, 5, 8), \\ (3, 2, 4), \ (3, 1, 5), \ (3, 1, 6), \ (3, 1, 5, 8), \\ (4, 8, 5), \ (4, 1, 6), \ (5, 2, 6), \ (5, 1, 3, 7), \\
  (6, 1, 3, 7), \ (6, 2, 5, 8), \ (7, 3, 1, 5, 8)
		\end{gathered}
		\end{equation*}
	\end{minipage}
	\\
	\vspace{5mm}
	\begin{minipage}{0.45\textwidth}
		\small
		\begin{equation*}
		\begin{split}
		w_{2,4} + w_{4,7}& \leq w_{2,3} + w_{3,7}\\
            w_{2,3} + w_{2,4}& \leq w_{1,3} + w_{1,4}\\
            w_{4,8} + w_{5,8}& \leq w_{2,4} + w_{2,5}\\
            w_{1,4} + w_{1,6}& \leq w_{2,4} + w_{2,6}\\
            w_{2,5} + w_{2,6}& \leq w_{1,5} + w_{1,6}\\
            w_{1,3} + w_{1,5} + w_{3,7} + w_{5,8}& \leq w_{4,7} + w_{4,8}
		\end{split}
		\end{equation*}
	\end{minipage}
	\hspace{5mm}$\implies $
	\begin{minipage}{0.2\textwidth}
		\small
		$$w_{5,8} \leq 0$$
	\end{minipage}
	\begin{center}
		\line(1,0){400}
	\end{center}
	\begin{minipage}[c]{0.4\textwidth}
		\centering
		\begin{tikzpicture}[scale=0.38, every node/.style={scale=0.38}]

		\node[draw,circle,minimum size=.5cm,inner sep=1pt] (1) at (-4,3.5) [scale=2]{$1$};
		\node[draw,circle,minimum size=.5cm,inner sep=1pt] (2) at (-6,0) [scale=2]{$2$};
		\node[draw,circle,minimum size=.5cm,inner sep=1pt] (3) at (-4,-3.5) [scale=2]{$3$};
		\node[draw,circle,minimum size=.5cm,inner sep=1pt] (5) at (4,-3.5) [scale=2]{$5$};
		\node[draw,circle,minimum size=.5cm,inner sep=1pt] (6) at (6,0) [scale=2]{$6$};
		\node[draw,circle,minimum size=.5cm,inner sep=1pt] (7) at (4,3.5) [scale=2]{$7$};
		\node[draw,circle,minimum size=.5cm,inner sep=1pt] (4) at ($(3)!.5!(5)$) [scale=2]{$4$};
		\node[draw,circle,minimum size=.5cm,inner sep=1pt] (8) at (-2,0) [scale=2]{$8$};
		\node[draw,circle,minimum size=.5cm,inner sep=1pt] (9) at (2,0) [scale=2]{$9$};

		\draw [line width=2pt,-] (1) -- (2);
		\draw [line width=2pt,-] (1) -- (7);
		\draw [line width=2pt,-] (1) -- (8);
		\draw [line width=2pt,-] (2) -- (3);
		\draw [line width=2pt,-] (3) -- (4);
		\draw [line width=2pt,-] (3) -- (8);
		\draw [line width=2pt,-] (4) -- (5);
		\draw [line width=2pt,-] (5) -- (6);
		\draw [line width=2pt,-] (5) -- (9);
		\draw [line width=2pt,-] (6) -- (7);
		\draw [line width=2pt,-] (7) -- (9);

		\end{tikzpicture}
	\end{minipage}
	\begin{minipage}[c]{0.4\textwidth}
		\centering
		\small
		\begin{equation*}
		\begin{gathered}
		 (1,8,3), \ (1,8,3,4) , \ (1,7,6,5), \ (1,7,6), \ (1,7,9), \\
		 (2,3,4), \ (2,3,4,5), \ (2,3,4,5,6), \ (2,1,7), \ (2,1,8), \  (2,1,7,9), \\
		 (3,4,5), \ (3,4,5,6), \ (3,8,1,7), \ (3,4,5,9), \\
		(4,5,6), \ (4,3,8,1,7), \ (4,3,8), \ (4,5,9), \  (5,6,7), \ (5,4,3,8), \\ 
		 (6,7,1,8), \ (6,7,9), \ (7,1,8), \ (8,3,4,5,9)
		\end{gathered}
		\end{equation*}
	\end{minipage}
	\\
	\vspace{5mm}
	\begin{minipage}{0.45\textwidth}
		\small
		\begin{equation*}
		\begin{split}
		w_{2,3} + w_{3,4} + w_{4,5} + w_{5,6} & \leq w_{1,2} + w_{1,7} + w_{6,7}\\
		w_{3,4} + w_{3,8} + w_{1,8} + w_{1,7} & \leq w_{4,5} + w_{5,6} + w_{6,7}\\
		w_{3,8} + w_{3,4} + w_{4,5} + w_{5,9} & \leq w_{1,8} + w_{1,7} + w_{7,9}\\
		w_{1,7} + w_{6,7} + w_{5,6} & \leq w_{1,8} + w_{3,8} + w_{3,4} + w_{4,5}\\
		w_{1,2} + w_{1,8} & \leq w_{2,3} + w_{3,8} \\
		w_{6,7} + w_{7,9} & \leq w_{5,6} + w_{5,9}
		\end{split}
		\end{equation*}
	\end{minipage}
	$\hspace{15mm}\implies $
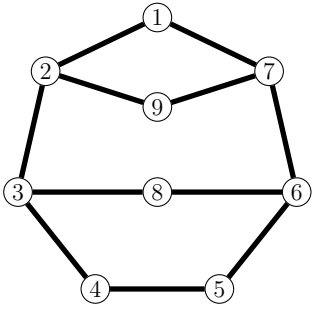
	\begin{minipage}{0.2\textwidth}
		\small
		$$w_{3,4} \leq 0$$
	\end{minipage}
	\begin{center}
		\line(1,0){400}
	\end{center}
\begin{minipage}[c]{0.4\textwidth}
		\centering
		\begin{tikzpicture}[scale=0.38, every node/.style={scale=0.38}]

            \node[draw,circle,minimum size=.5cm,inner sep=1pt] (1) at (0*360/7 +90: 5cm)  [scale=2]{$1$};
            \node[draw,circle,minimum size=.5cm,inner sep=1pt](2) at (1*360/7 +90: 5cm)  [scale=2]{$2$};
            \node[draw,circle,minimum size=.5cm,inner sep=1pt](3) at (2*360/7 +90: 5cm)  [scale=2]{$3$};
            \node[draw,circle,minimum size=.5cm,inner sep=1pt] (4) at (3*360/7 +90: 5cm)  [scale=2]{$4$};
            \node[draw,circle,minimum size=.5cm,inner sep=1pt] (5) at (4*360/7 +90: 5cm)  [scale=2]{$5$};
            \node[draw,circle,minimum size=.5cm,inner sep=1pt] (6) at (5*360/7 +90: 5cm)  [scale=2]{$6$};
            \node[draw,circle,minimum size=.5cm,inner sep=1pt](7) at (6*360/7 +90: 5cm)  [scale=2]{$7$};
            \node[draw,circle,minimum size=.5cm,inner sep=1pt](8) at ($(3)!.5!(6)$)  [scale=2]{$8$};
            \node[draw,circle,minimum size=.5cm,inner sep=1pt](9) at ($(2)!.5!(7) - (0,1.25)$)  [scale=2]{$9$};

            \draw [line width=2pt,-] (1) -- (2);
            \draw [line width=2pt,-] (1) -- (7);
            \draw [line width=2pt,-] (2) -- (3);
            \draw [line width=2pt,-] (2) -- (9);
            \draw [line width=2pt,-] (3) -- (4);
            \draw [line width=2pt,-] (3) -- (8);
            \draw [line width=2pt,-] (4) -- (5);
            \draw [line width=2pt,-] (5) -- (6);
            \draw [line width=2pt,-] (6) -- (7);
            \draw [line width=2pt,-] (6) -- (8);
            \draw [line width=2pt,-] (7) -- (9);		
		
		\end{tikzpicture}
	\end{minipage}
	\begin{minipage}[c]{0.4\textwidth}
		\centering
		\small
		\begin{equation*}
		\begin{gathered}
		(1, 2, 3), \ (1, 2, 3, 4), \ (1, 2, 3, 4, 5), \ (1, 7, 6), \ (1, 7, 6, 8), \\ (1, 2, 9), \ (2, 3, 4), \ (2, 3, 4, 5), \ (2, 9, 7, 6), \ (2, 9, 7), \ (2, 3, 8), \\ (3, 4, 5), \ (3, 8, 6), \ (3, 2, 9, 7), \ (3, 2, 9), \ (4, 5, 6), \\ (4, 5, 6, 7), \ (4, 5, 6, 8), \ (4, 5, 6, 7, 9), \ (5, 6, 7), \\ (5, 6, 8), \ (5, 6, 7, 9), \ (6, 7, 9), \ (7, 6, 8), \ (8, 3, 2, 9)
		\end{gathered}
		\end{equation*}
	\end{minipage}
	\\
	\vspace{5mm}
	\begin{minipage}{0.45\textwidth}
		\small
		\begin{equation*}
		\begin{split}
w_{1,2} + w_{2,3} + w_{3,4} + w_{4,5}& \leq w_{1,7} + w_{6,7} + w_{5,6}\\
w_{4,5} + w_{5,6} + w_{6,7} + w_{7,9}& \leq  w_{2,3} + w_{3,4} + w_{2,9}\\
w_{1,7} + w_{6,7} + w_{6,8}& \leq w_{1,2} + w_{2,3} + w_{3,8}\\
w_{2,3} + w_{2,9} + w_{3,8}& \leq w_{6,8} +w_{6,7} + w_{7,9}
		\end{split}
		\end{equation*}
	\end{minipage}
	$\hspace{15mm}\implies $
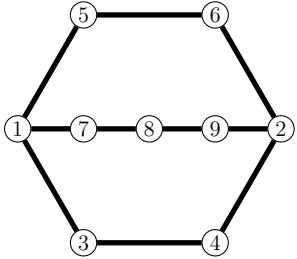
	\begin{minipage}{0.2\textwidth}
		\small
		$$w_{4,5} \leq 0$$
	\end{minipage}
	\begin{center}
		\line(1,0){400}
	\end{center}
	\begin{minipage}[c]{0.4\textwidth}
		\centering
		\begin{tikzpicture}[scale=0.35, every node/.style={scale=0.35}]

		\node[draw,circle,minimum size=.5cm,inner sep=1pt] (1) at (0*360/6 +180: 5cm) [scale =2] {$1$};
		\node[draw,circle,minimum size=.5cm,inner sep=1pt] (2) at (3*360/6 +180: 5cm) [scale =2] {$2$};
		\node[draw,circle,minimum size=.5cm,inner sep=1pt] (3) at (1*360/6 +180: 5cm) [scale =2] {$3$};
		\node[draw,circle,minimum size=.5cm,inner sep=1pt] (4) at (2*360/6 +180: 5cm)  [scale =2]{$4$};
		\node[draw,circle,minimum size=.5cm,inner sep=1pt] (5) at (5*360/6 +180: 5cm)  [scale =2]{$5$};
		\node[draw,circle,minimum size=.5cm,inner sep=1pt] (6) at (4*360/6 +180: 5cm)  [scale =2]{$6$};
		\node[draw,circle,minimum size=.5cm,inner sep=1pt] (7) at ($(1)!.25!(2)$)  [scale =2]{$7$};
		\node[draw,circle,minimum size=.5cm,inner sep=1pt] (8) at ($(1)!.5!(2)$)  [scale =2]{$8$};
		\node[draw,circle,minimum size=.5cm,inner sep=1pt] (9) at ($(1)!.75!(2)$)  [scale =2]{$9$};

		\draw [line width=2pt,-] (1) -- (3);
		\draw [line width=2pt,-] (1) -- (5);
		\draw [line width=2pt,-] (1) -- (7);
		\draw [line width=2pt,-] (2) -- (4);
		\draw [line width=2pt,-] (2) -- (6);
		\draw [line width=2pt,-] (2) -- (9);
		\draw [line width=2pt,-] (3) -- (4);
		\draw [line width=2pt,-] (5) -- (6);
		\draw [line width=2pt,-] (7) -- (8);
		\draw [line width=2pt,-] (8) -- (9);

		\end{tikzpicture}
	\end{minipage}
	\begin{minipage}[c]{0.4\textwidth}
		\centering
		\small
		\begin{equation*}
		\begin{gathered}
		(1,7,8,9,2), \  (1,3,4),  \ (1,5,6),  \ (1,7,8), \ (1,7,8,9), \\
		(2,4,3), \  (2,6,5),\ (2,9,8,7), \ (2,9,8),   \\
		(3,1,5), \ (3,1,5,6), \ (3,1,7), \ (3,1,7,8), \ (3,4,2,9), \ (4,2,6,5), \ (4,2,6), \\
		(4,3,1,7), \ (4,3,1,7,8), \ (4,2,9), \ (5,1,7), \ (5,1,7,8), \\
		(5,1,7,8,9), \ (6,2,9,8,7), \ (6,2,9,8), \ (6,2,9),  \ (7,8,9)
		\end{gathered}
		\end{equation*}
	\end{minipage}
	\\
	\vspace{5mm}
	\begin{minipage}{0.45\textwidth}
		\small
		\begin{equation*}
		\begin{split}
		w_{2,6} + w_{2,9} + w_{8,9} + w_{7,8} & \leq w_{5,6} + w_{1,5} + w_{1,7}\\
		w_{3,4} + w_{1,3} + w_{1,7} + w_{7,8} & \leq w_{2,4} + w_{2,9} + w_{8,9}\\
		w_{1,5} + w_{1,7} + w_{7,8} + w_{8,9} & \leq w_{5,6} + w_{2,6} + w_{2,9}\\
		w_{2,4} + w_{2,6} + w_{5,6}  & \leq w_{3,4} + w_{1,3} + w_{1,5}\\
		w_{1,3} + w_{1,5} + w_{5,6}  & \leq w_{3,4} + w_{2,4} + w_{2,6}\\
		w_{3,4} + w_{2,4} + w_{2,9}  & \leq w_{1,3} + w_{1,7} + w_{7,8} + w_{8,9}
		\end{split}
		\end{equation*}
	\end{minipage}
	$\hspace{15mm}\implies $
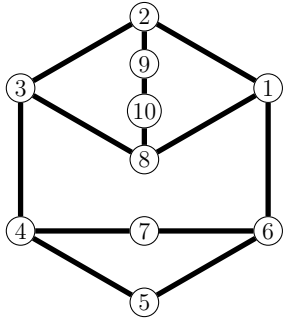
	\begin{minipage}{0.2\textwidth}
		\small
		$$w_{7,8} \leq 0$$
	\end{minipage}
	\begin{center}
		\line(1,0){400}
	\end{center}
\begin{minipage}[c]{0.4\textwidth}
		\centering
		\begin{tikzpicture}[scale=0.38, every node/.style={scale=0.38}]
            \node[draw,circle,minimum size=.5cm,inner sep=1pt] (1) at (0*360/6 +30: 5cm) [scale=2]{$1$};
            \node[draw,circle,minimum size=.5cm,inner sep=1pt] (2) at (1*360/6 +30: 5cm) [scale=2]{$2$};
            \node[draw,circle,minimum size=.5cm,inner sep=1pt] (3) at (2*360/6 +30: 5cm) [scale=2]{$3$};
            \node[draw,circle,minimum size=.5cm,inner sep=1pt](4) at (3*360/6 +30: 5cm) [scale=2]{$4$};
            \node[draw,circle,minimum size=.5cm,inner sep=1pt](5) at (4*360/6 +30: 5cm) [scale=2]{$5$};
            \node[draw,circle,minimum size=.5cm,inner sep=1pt] (6) at (5*360/6 +30: 5cm) [scale=2]{$6$};
            \node[draw,circle,minimum size=.5cm,inner sep=1pt](7) at ($(4)!.5!(6)$) [scale=2]{$7$};
            \node[draw,circle,minimum size=.5cm,inner sep=1pt] (8) at (0,0) [scale=2]{$8$};
            \node[draw,circle,minimum size=.5cm,inner sep=1pt](9) at ($(2)!.33!(8)$) [scale=2]{$9$};
            \node[draw,circle,minimum size=.5cm,inner sep=1pt] (10) at ($(2)!.66!(8)$) [scale=2]{$10$};

            \draw [line width=2pt,-] (1) -- (2);
            \draw [line width=2pt,-] (1) -- (6);
            \draw [line width=2pt,-] (1) -- (8);
            \draw [line width=2pt,-] (2) -- (3);
            \draw [line width=2pt,-] (2) -- (9);
            \draw [line width=2pt,-] (3) -- (4);
            \draw [line width=2pt,-] (3) -- (8);
            \draw [line width=2pt,-] (4) -- (5);
            \draw [line width=2pt,-] (4) -- (7);
            \draw [line width=2pt,-] (5) -- (6);
            \draw [line width=2pt,-] (6) -- (7);
            \draw [line width=2pt,-] (8) -- (10);
            \draw [line width=2pt,-] (9) -- (10);	
		
		\end{tikzpicture}
	\end{minipage}
	\begin{minipage}[c]{0.4\textwidth}
		\centering
		\small
		\begin{equation*}
		\begin{gathered}
		(1, 2, 3), \ (1, 2, 3, 4), \ (1, 6, 5), \ (1, 6, 7), \ (1, 2, 9), \\ (1, 2, 9, 10), \ (2, 3, 4), \ (2, 1, 6, 5), \ (2, 1, 6), \ (2, 1, 6, 7), \\ (2, 3, 8), \ (2, 9, 10), \ (3, 4, 5), \ (3, 4, 5, 6), \ (3, 4, 7), \\ (3, 8, 10, 9), \ (3, 8, 10), \ (4, 5, 6), \ (4, 3, 8), \ (4, 3, 8, 10, 9), \\ (4, 3, 8, 10), \ (5, 4, 7), \ (5, 6, 1, 8), \ (5, 6, 1, 2, 9), \\ (5, 6, 1, 2, 9, 10), \ (6, 1, 8), \ (6, 1, 2, 9), \ (6, 1, 2, 9, 10), \\ (7, 4, 3, 8), \ (7, 6, 1, 2, 9), \ (7, 6, 1, 2, 9, 10), \ (8, 10, 9)
		\end{gathered}
		\end{equation*}
	\end{minipage}
	\\
	\vspace{5mm}
	\begin{minipage}{0.45\textwidth}
		\small
		\begin{equation*}
		\begin{split}
  w_{6,7}+ w_{1,6}  +w_{1,2} + w_{2,9} + w_{9,10}& \leq   w_{4,7} +w_{3,4}+ w_{3,8} + w_{8,10}\\
  w_{3,8} + w_{8,10} + w_{9,10}& \leq w_{2,3} + w_{2,9}\\
  w_{3,4} + w_{4,5} + w_{5,6}& \leq  w_{2,3}+w_{1,2} + w_{1,6}\\
   w_{5,6}+w_{1,6} + w_{1,8}& \leq w_{4,5}+w_{3,4} + w_{3,8}\\
w_{1,2} + w_{2,3} + w_{3,4}& \leq w_{1,6} + w_{5,6} + w_{4,5}\\
w_{2,3} + w_{3,8}& \leq w_{1,2} + w_{1,8}\\
w_{4,5} + w_{4,7}& \leq w_{5,6} + w_{6,7}
		\end{split}
		\end{equation*}
	\end{minipage}
	$\hspace{25mm}\implies $
	\begin{minipage}{0.2\textwidth}
		\small
		$$w_{9,10} \leq 0$$
	\end{minipage}
	\begin{center}
		\line(1,0){400}
	\end{center}
 \newpage
 \printbibliography
\end{document}